\documentclass[reqno,12pt]{amsart}

\usepackage{color} 
\usepackage{hyperref}
\hypersetup{
    unicode=false,          
    pdftoolbar=true,        
    pdfmenubar=true,        
    pdffitwindow=false,     
    pdfstartview={FitH},    
    pdftitle={MCP Article},      
    pdfauthor={Michael Holst},   
    pdfsubject={Mathematics},    
    pdfcreator={Michael Holst},  
    pdfproducer={Michael Holst}, 
    pdfkeywords={PDE, analysis, mathematical physics}, 
    pdfnewwindow=true,      
    colorlinks=true,        
    linkcolor=red,          
    citecolor=blue,         
    filecolor=magenta,      
    urlcolor=cyan           
}
\usepackage{times}
\usepackage{amsfonts}
\usepackage{amsmath}
\usepackage{amsthm}
\usepackage{amssymb}
\usepackage{amsbsy}
\usepackage{amscd}
\usepackage{enumerate}

\newtheorem{theorem}{Theorem}[section]
\newtheorem{corollary}[theorem]{Corollary}
\newtheorem{lemma}[theorem]{Lemma}
\newtheorem{proposition}[theorem]{Proposition}
\newtheorem{definition}[theorem]{Definition}
\newtheorem{remark}[theorem]{Remark}

\numberwithin{equation}{section}

\newcommand{\reals}{\mathbb{R}}

\newcommand{\lab}{\label}

\newcommand{\floor}[1]{\lfloor{#1}\rfloor}

\definecolor{mve}{rgb}{0.7,0.35,0.15}
\definecolor{brght}{rgb}{0.825,0.2625,0.15}
\definecolor{yello}{rgb}{1,0.925,0.65}
\definecolor{bluu}{rgb}{0.65, 0.95, 1}
\definecolor{bluu2}{rgb}{0.2, 0.5, 0.8}

\DeclareMathOperator*{\esssup}{ess\,sup}

\newenvironment{itemizeXX}
{\begin{list}{$\circ$}
 {\setlength{\leftmargin}{1.5em}
  \setlength{\topsep}{0.5em}
  \setlength{\itemsep}{0.5em}
  \setlength{\labelwidth}{50.0em}}}
 {\end{list}}

 \newenvironment{itemizeXALI}
{\begin{list}{\labelitemi}
 {\setlength{\leftmargin}{0.4em}
  \setlength{\topsep}{0.5em}
  \setlength{\itemsep}{0.5em}
  \setlength{\labelwidth}{50.0em}}}
 {\end{list}}
\newenvironment{itemizeXXALI}
{\begin{list}{$\circ$}
 {\setlength{\leftmargin}{0.5em}
  \setlength{\topsep}{0.5em}
  \setlength{\itemsep}{0.5em}
  \setlength{\labelwidth}{50.0em}}}
 {\end{list}}
\begin{document}

\title[Multiplication in Sobolev Spaces]
      {Multiplication in Sobolev Spaces, Revisited}

\author[A. Behzadan]{A. Behzadan}
\email{a.behzadan@csus.edu}
\address{Department of Mathematics and Statistics\\
        California State University Sacramento\\
        Sacramento CA 95819}

\author[M. Holst]{M. Holst}
\email{mholst@math.ucsd.edu}
\address{Department of Mathematics\\
        University of California San Diego\\
        La Jolla CA 92093}

\thanks{AB was supported by NSF Award~1262982.}
\thanks{MH was supported in part by
        NSF Awards~1262982, 1318480, and 1620366.}

\date{\today}

\begin{abstract}
In this article, we re-examine some of the classical pointwise
multiplication theorems in Sobolev-Slobodeckij spaces, in part motivated by a simple counter-example that illustrates how
certain multiplication theorems fail in Sobolev-Slobodeckij
spaces when a bounded domain is replaced by $\mathbb{R}^n$. We
identify the source of the failure, and examine why the same
failure is not encountered in Bessel potential spaces. To analyze the situation, we begin with a survey of the classical
multiplication results stated and proved in the 1977 article of
Zolesio, and carefully distinguish between the case of spaces
defined on the all of $\mathbb{R}^n$ and spaces defined on a
bounded domain (with e.g.\ a Lipschitz boundary). However, the
survey we give has a few new wrinkles; the proofs we include are
based almost exclusively on interpolation theory rather than
Littlewood-Paley theory and Besov spaces, and some of the results
we give and their proofs, including the results for negative
exponents, do not appear in the literature in this form. We also include a particularly important variation of one
of the multiplication theorems that is relevant to the study of
nonlinear PDE systems arising in general relativity and other
areas. The conditions for multiplication to be continuous in the
case of Sobolev-Slobodeckij spaces are somewhat subtle and
intertwined, and as a result, the multiplication theorems of
Zolesio in 1977 have been cited (more than once) in the standard
literature in slightly more generality than what is actually
proved by Zolesio, and in cases that allow for the construction
of counter-examples such as the one included here.
\end{abstract}

\maketitle
\tableofcontents

\clearpage

\section{Introduction}
 \label{sec:intro}

Let $f\in W^{s_1,p_1}$ and $g\in W^{s_2,p_2}$.
What can be said about the product $fg$?
In particular, to which Sobolev space $W^{s,p}$ does the product $fg$ belong? The answer to this question plays a key role in a number of applications
in analysis, and having a very complete answer to this question is critical
to the modern theory of partial differential equations (PDE).
It is particularly important to the understanding of the solution theory
for elliptic PDE, and it is also central to the design and analysis of
approximations of their solutions using various techniques.
In the modern theory of partial differential equations, PDE are interpreted as equations of the form $Au=f$
where $A$ is an operator between suitable function spaces. In
this view, the existence of a unique solution for all right hand
sides is equivalent to $A$ being bijective. A main difficulty is
in choosing the domain of realization and the codomain of the operator $A$, that is, choosing appropriate function spaces $X$ and $Y$ such that
\begin{enumerate}
\item $A$ can be considered as an operator from $X$ to $Y$ and $f\in
Y$, i.e., we need to ensure that the equation makes sense if we
consider $X$ and $Y$ as the domain and codomain of $A$.
\item $A$ (or a family of approximations of $A$) has ``nice"
properties as an operator (or a family of operators) from $X$ to
$Y$. Here ``nice properties" may refer to any of the
following:
$A$ is continuous, $A$ is compact, $A$ is Fredholm,
$A$ is injective, $A$ is surjective , $A$ satisfies a maximum
principle, etc.
\end{enumerate}
As it turns out, for elliptic equations, using Sobolev spaces (or
weighted Sobolev spaces) as domain and codomain of $A$ helps us to
ensure that $A$ has ``nice" properties. But how to determine
appropriate Sobolev spaces to make sure that the equation makes
sense? This is one of the applications where pointwise multiplication
theorems are particularly important.
The best way to see this is by looking at
a very simple example. Consider the equation $-\Delta u+V u=f$ in
$\Omega \subseteq \mathbb{R}^n$ where $\Omega$ is a domain with smooth boundary. Suppose we want to seek the
unknown function $u$ in the Sobolev space $W^{s,p}$ where $s\geq 2$ and $1<p<\infty$. Having this
assumption, what restrictions do we need to impose on the data
$V$ and $f$? The assumption  $u\in W^{s,p}$ implies that $-\Delta
u \in W^{s-2,p}$. Therefore for the equation to make sense (as an
equality in $W^{s-2,p}$), $f$ and $Vu$ must belong to
$W^{s-2,p}$. So now we need to find those Sobolev spaces
$W^{r,q}$ such that if $V\in W^{r,q}$, then $Vu\in W^{s-2,p}$.
That is, we need to find those exponents $r$ and $q$ for which the
product of a function in $W^{r,q}$ and a function in $W^{s,p}$
belongs to $W^{s-2,p}$.
If one now considers even the simplest nonlinear generalization
of this problem, say $-\Delta u+V u^p=f$
in $\Omega \subseteq \mathbb{R}^n$, then it is immediately clear
that the conditions on the spaces become substantially more complicated,
and multiplication theorems are a critical tool in the
analysis of nonlinear PDE.

There are a number of articles and book chapters that are devoted
to the study of pointwise multiplication in function spaces, e.g.
\cite{jZ77,37}. Unfortunately most references study the question
in the general setting of Triebel-Lizorkin spaces and use
technical tools from Littlewood-Paley theory and theory of Besov
spaces to prove the results. A main feature of this article is
that the key results are proved without any direct reference to
Littlewood-Paley theory and Besov spaces, which makes it
accessible to a wider range of readers. In particular, we give
alternative proofs for a number of results first stated in
\cite{jZ77} for Sobolev spaces with nonnegative exponents. Additionally, we extend those results to Sobolev spaces with negative
exponents. We clearly distinguish between the case of Sobolev
spaces defined on the entire space $\reals^n$ and the case where
Sobolev spaces are defined on a bounded domain.
Lastly, we remark that one of the main tools we use throughout
the paper, namely interpolation theory, is a fascinating topic itself;
we only briefly summarize some of the main ideas and results we need
in the paper in Section~\ref{subsec:importantproperties}.
Much more complete expositions can be found, for example, in~\cite{36}.


{\bf\em Outline of the Paper.} A brief outline of the remainder of the paper is as follows. In Section
\ref{app:spacesbasicdef} we review some of the basic
well-known definitions and facts about Sobolev spaces. In Section
\ref{subsec:importantproperties} summarize some basic facts about interpolation theory and several key properties of Sobolev spaces.
 In Section \ref{subsec:maintheorems} we review a counter-example for
generalized Holder-type inequalities in Sobolev-Slobodeckij
spaces. In sections \ref{sec:bessel}, \ref{sec:integers},
\ref{sec:positives}, and \ref{sec:negatives} we state and prove
the main theorems.

\section{Notation and Definitions}
\label{app:spacesbasicdef}

In this section we briefly review some basic notation and
definitions related to the Sobolev spaces,
with emphasis on fractional order spaces.
Throughout the manuscript we use the notation
$A \preceq B$ to mean $A\leq cB$, where $c$ is a positive constant
that does not depend on the non-fixed parameters appearing in $A$ and $B$.
We use the notation $X\hookrightarrow Y$ to mean $X
\subseteq Y$ and the inclusion map is continuous.


\begin{definition}
Let $k\in \mathbb{N}_0$, $1< p<\infty$. The Sobolev space
$W^{k,p}(\mathbb{R}^n)$ is defined as follows:
\begin{equation*}
W^{k,p}(\mathbb{R}^n)=\{u\in L^p (\mathbb{R}^n):
\|u\|_{W^{k,p}(\mathbb{R}^n)}:=\sum_{|\nu|\leq
k}\|\partial^{\nu}u\|_p<\infty\}
\end{equation*}
For $k\in \mathbb{N}$, the Sobolev space $W^{-k,p}(\mathbb{R}^n)$
is defined as the topological dual of $W^{k,p'}(\mathbb{R}^n)$ where $
\frac{1}{p}+\frac{1}{p'}=1$. That is
$W^{-k,p}(\mathbb{R}^n):=(W^{k,p'}(\mathbb{R}^n))^{*}$ .
\end{definition}
\begin{remark}
\leavevmode
\begin{itemize}
\item For real-valued function $u(x_1,\dots,x_n)$ and multi-index $\nu=(\nu_1,\dots,\nu_n) \in \mathbb{N}_0^n$,
\begin{equation*}
|\nu|:=\nu_1+\dots+\nu_n, \quad
\partial^{\nu}u:=\frac{\partial^{|\nu|}u}{\partial
x_1^{\nu_1}\dots\partial x_n^{\nu_n}},\quad
\|u\|_p:=(\int_{\reals^n}|u|^p dx)^{\frac{1}{p}}\,.
\end{equation*}
\item The Sobolev norm is defined so that $\partial^{\alpha}: W^{k,p}(\mathbb{R}^n)\rightarrow W^{k-|\alpha|,p}(\mathbb{R}^n)$
becomes a continuous operator for $|\alpha|\leq k$. It can be shown that
$C_c^{\infty}(\mathbb{R}^n)$ is dense in $W^{k,p}(\mathbb{R}^n)$.
In fact, $W^{k,p}(\mathbb{R}^n)$ is the completion of the space of
smooth functions with respect to
$\|\cdot\|_{W^{k,p}(\mathbb{R}^n)}$.
\item Clearly, if $k_1\geq k_0$, then $W^{k_1,p}(\mathbb{R}^n)\subseteq
W^{k_0,p}(\mathbb{R}^n)$.
\end{itemize}
\end{remark}
 There are \textbf{nonequivalent} ways to generalize the above
definition to allow noninteger exponents. We can define Sobolev
spaces with noninteger exponents as
\begin{enumerate}
\item Slobodeckij spaces, or,
\item Bessel potential spaces.
\end{enumerate}
There are at least three \textbf{equivalent} methods to define each of the
above spaces:
\begin{enumerate}
\item Classical definition
\item Definition based on interpolation theory
\item Definition based on Littlewood-Paley theory
\end{enumerate}
\textbf{1-Classical Definitions}
\begin{definition}\lab{def8201}
Let $s\in \mathbb{R}$ and $p\in [1,\infty]$. The
Sobolev-Slobodeckij space $W^{s,p}(\mathbb{R}^n)$ is defined as
follows:
\begin{itemize}
\item If $s=k\in \mathbb{N}_0$, $p\in[1,\infty]$,
\begin{equation*}
W^{k,p}(\mathbb{R}^n)=\{u\in L^p (\mathbb{R}^n):
\|u\|_{W^{k,p}(\mathbb{R}^n)}:=\sum_{|\nu|\leq
k}\|\partial^{\nu}u\|_p<\infty\}
\end{equation*}
\item If $s=\theta\in(0,1)$, $p\in[1,\infty)$,
\begin{equation*}
W^{\theta,p}(\mathbb{R}^n)=\{u\in L^p (\mathbb{R}^n):
 |u|_{W^{\theta,p}(\mathbb{R}^n)}:=\big(\int\int_{\mathbb{R}^n\times
\mathbb{R}^n}\frac{|u(x)-u(y)|^p}{|x-y|^{n+\theta p}}dx
dy\big)^{\frac{1}{p}} <\infty\}
\end{equation*}
\item If $s=\theta\in(0,1)$, $p=\infty$,
\begin{equation*}
W^{\theta,\infty}(\mathbb{R}^n)=\{u\in L^{\infty} (\mathbb{R}^n):
 |u|_{W^{\theta,\infty}(\mathbb{R}^n)}:=\esssup_{x,y \in
\mathbb{R}^n, x\neq y}\frac{|u(x)-u(y)|}{|x-y|^{\theta}} <\infty\}
\end{equation*}
\item If $s=k+\theta,\, k\in \mathbb{N}_0,\, \theta\in(0,1)$,
$p\in[1,\infty]$,
\begin{equation*}
W^{s,p}(\mathbb{R}^n)=\{u\in
W^{k,p}(\mathbb{R}^n):\|u\|_{W^{s,p}(\mathbb{R}^n)}:=\|u\|_{W^{k,p}(\mathbb{R}^n)}+\sum_{|\nu|=k}
|\partial^{\nu}u|_{W^{\theta,p}(\mathbb{R}^n)}<\infty\}
\end{equation*}
\item If $s<0$ and $p\in(1,\infty)$,
\begin{equation*}
W^{s,p}(\mathbb{R}^n)=(W^{-s,p'}(\mathbb{R}^n))^{*} \quad
(\frac{1}{p}+\frac{1}{p'}=1).
\end{equation*}
\end{itemize}
\end{definition}
 Alternatively, for $s\in \mathbb{R}$ and $1<p<\infty$, one can
define Sobolev spaces as Bessel potential spaces
$H^{s,p}(\mathbb{R}^n)$:
\begin{equation*}
H^{s,p}(\mathbb{R}^n)= \{u\in S'(\mathbb{R}^n):
\|u\|_{W^{s,p}(\mathbb{R}^n)}:=\|\mathcal{F}^{-1}(\langle\xi\rangle^s\mathcal{F}u)\|_{L^p}<\infty\}\,,
\end{equation*}
where $\langle\xi\rangle:=(1+|\xi|^2)^{\frac{1}{2}}$. Here
$\mathcal{F}$ denotes the Fourier transform on the space
$S'(\mathbb{R}^n)$ of tempered distributions. It is a well-known
fact that $H^{s,p}(\mathbb{R}^n)=(H^{-s,p'}(\mathbb{R}^n))^{*}$
and for $k\in \mathbb{Z}$ the two definitions agree \cite{31,36,
StHo2011a}. Also for $s\in \mathbb{R}$ and $p=2$ the two
definitions
agree \cite{31,StHo2011a}. $H^{s,2}(\reals^n)$ is often denoted by $H^s(\reals^n)$.\\
\textbf{2-Definitions Based on Interpolation Theory}\\ A short
introduction to interpolation theory in Banach spaces is given in
Section \ref{subsec:importantproperties}. Suppose $s\in\reals
\setminus \mathbb{Z}$, $1<p<\infty$, and let
$\theta:=s-\floor{s}$.
\begin{itemize}
\item
$W^{s,p}(\mathbb{R}^n)=(W^{\floor{s},p}(\mathbb{R}^n),W^{\floor{s}+1,p}(\mathbb{R}^n))_{\theta,p}$.
\item  $H^{s,p}(\mathbb{R}^n)=[H^{\floor{s},p}(\mathbb{R}^n),H^{\floor{s}+1,p}(\mathbb{R}^n)]_{\theta}$.
\end{itemize}
\textbf{3-Definitions Based on Littlewood-Paley Theory}\\
 Consider an open cover of $\mathbb{R}^n$ that consists of the
following sets (annuli):
\begin{equation*}
B_2,\quad B_4\setminus \bar{B_1}, \quad B_8\setminus
\bar{B_2},\,...,\,B_{2^{j+1}}\setminus \bar{B}_{2^{j-1}},
\end{equation*}
where $B_r$ is the open ball of radius $r$ centered at the
origin. Consider the following smooth partition of unity subordinate to
the above cover of $\mathbb{R}^n$:
\begin{align*}
&\varphi_0=1\quad \textrm{on} \quad B_1,\quad {\rm supp} \varphi_0\subseteq B_2,\\
&\varphi(\xi)= \varphi_0(\xi)-\varphi_0(2\xi) \quad ({\rm supp}
\varphi \subseteq B_2, \quad \varphi=0 \, \textrm{on} \,
B_{\frac{1}{2}}),\\
& \forall j\geq 1 \quad \varphi_j(\xi)= \varphi(2^{-j}\xi).
\end{align*}
One can easily check that $\sum_{j=0}^{\infty}\varphi_j (\xi)=1$.
\begin{definition}
${}$
\begin{itemize}
\item For $s\in \mathbb{R}$, $1\leq p<\infty$, and $1 \leq q<\infty$ (or
$p=q=\infty$) we define the Triebel-Lizorkin space
$F^{s}_{p,q}(\mathbb{R}^n)$ as follows
\begin{equation*}
F^{s}_{p,q}(\mathbb{R}^n)=\{u\in S'(\mathbb{R}^n): \|
u\|_{F^{s}_{p,q}(\mathbb{R}^n)}= \big |\big| \parallel
2^{sj}\mathcal{F}^{-1}(\varphi_j \mathcal{F}u)\parallel_{l^q}
\big|\big|_{L^p(\mathbb{R}^n)}<\infty\}
\end{equation*}
\item For $s\in \mathbb{R}$, $1\leq p<\infty$, and $1 \leq
q<\infty$ we define the Besov space $B^s_{p,q}(\mathbb{R}^n)$ as
follows
\begin{equation*}
B^{s}_{p,q}(\mathbb{R}^n)=\{u\in S'(\mathbb{R}^n): \|
u\|_{B^{s}_{p,q}(\mathbb{R}^n)}= \big |\big| \parallel
2^{sj}\mathcal{F}^{-1}(\varphi_j
\mathcal{F}u)\parallel_{L^p(\mathbb{R}^n)}
\big|\big|_{l^q}<\infty\}
\end{equation*}
\end{itemize}
\end{definition}
 We have the following relations \cite{36,StHo2011a,38}
\begin{itemize}
\item $L^p=F^0_{p,2},\quad 1<p<\infty$.
\item $B^{s}_{p,p}=F^{s}_{p,p}, \quad s\in \mathbb{R},\,\, 1<p<\infty$.
\item $H^{s,p}=F^{s}_{p,2}, \quad s\in \mathbb{R},\,\, 1<p<\infty$.
\item $W^{k,p}=H^{k,p}=F^{k}_{p,2}, \quad k\in \mathbb{Z},\,\,
1<p<\infty$.
\item $W^{s,p}=B^s_{p,p}=F^s_{p,p},\quad s\in \mathbb{R}\setminus
\mathbb{Z},\,\, 1<p<\infty$.
\item If $k\in \mathbb{N}$, then $B^{k}_{p,p}\hookrightarrow W^{k,p}$ for $1\leq p \leq 2$
 and $W^{k,p}\hookrightarrow B^k_{p,p}$ for $p\geq 2$.
\end{itemize}
\begin{definition}
Let $\Omega$ be an open bounded subset of $\mathbb{R}^n$ with
Lipschitz continuous boundary. Suppose $s\geq 0$ and $1\leq
p<\infty$. $W^{s,p}(\Omega)$ is defined as the restriction of
$W^{s,p}(\mathbb{R}^n)$ to $\Omega$ and is equipped with the
following norm:
\begin{equation*}
\|u\|_{W^{s,p}(\Omega)}=\inf_{v\in W^{s,p}(\mathbb{R}^n),
v|_{\Omega}=u}\|v\|_{W^{s,p}(\mathbb{R}^n)}.
\end{equation*}
$W^{s,p}_0(\Omega)$ is defined as the closure of
$C_c^{\infty}(\Omega)$ in $W^{s,p}(\Omega)$. $W^{s,2}_0(\Omega)$
is often denoted by $\mathring{H}^s(\Omega)$.
\end{definition}
\begin{remark}
${}$
\begin{itemize}
\item One may define $H^{s,p}(\Omega)$, $B^s_{p,q}(\Omega)$, and
$F^s_{p,q}(\Omega)$ in a similar fashion.
\item It can be shown that for $k\in \mathbb{N}_0$ and $1<p<\infty$ the above
definition of $W^{k,p}(\Omega)$ agrees with the following
intrinsic definition \cite{36}
\begin{equation*}
W^{k,p}(\Omega)=\{u\in L^p (\Omega):
\|u\|_{W^{k,p}(\Omega)}:=\sum_{|\nu|\leq
k}\|\partial^{\nu}u\|_{L^p(\Omega)}<\infty\}.
\end{equation*}
Indeed, for all $s>0$ and $1<p<\infty$, the Sobolev space $W^{s,p}(\Omega)$ can be equivalently defined using intrinsic norms analogous to those displayed in Definition \ref{def8201}.
\end{itemize}
\end{remark}
For $s<0$ and $1<p<\infty$ we define
$W^{s,p}(\Omega):=(W^{-s,p'}_0(\Omega))^{*}$.

When there is no danger of ambiguity about the domain we may write
\begin{itemize}
\item $W^{s,p}$ instead of $W^{s,p}(\Omega)$,
\item $\| .\|_{W^{s,p}}$ or $\|\cdot\|_{s,p}$ instead of
$\|\cdot\|_{W^{s,p}(\Omega)}$.
\end{itemize}
\section{Key Properties of Sobolev Spaces}
 \label{subsec:importantproperties}
 We begin with reviewing the basic definitions of
 interpolation theory in Banach spaces. A detailed discussion can be found in \cite{36}.

 A pair $\{A_0, A_1\}$ of two Banach spaces is said to be an
\emph{interpolation couple}, if both spaces are continuously
embedded in a common Hausdorff topological vector space $A$.
 We may consider the following two subspaces:
\begin{itemize}
\item $A_0 \cap A_1$, and
\item $A_0+A_1:= \{a\in A : \exists a_0 \in A_0,\,\, \exists a_1\in A_1,\,\,
a=a_0+a_1\}$.
\end{itemize}
 Equipped with the norms
\begin{align*}
&\|a\|_{A_0 \cap A_1}:=\min\{\|a\|_{A_0}, \|a\|_{A_1}\}\\
&\|a\|_{A_0+A_1}:=\inf\{\|a_0\|_{A_0}+\|a_1\|_{A_1}:
a=a_0+a_1\}\quad 
\end{align*}
$A_0 \cap A_1$ and $A_0+A_1$ become Banach spaces. \emph{Real
interpolation} and \emph{complex interpolation} are two,
generally nonequivalent, methods for constructing intermediate
spaces between $A_0$ and $A_1$ in the sense that the new space
lies between $A_0\cap A_1$ and $A_0+A_1$ (with continuous
injections).
\begin{itemize}
\item Given a pair $(\theta,p)$ with $0< \theta <1$ and
$1<p<\infty$, the real interpolation functor constructs an
intermediate Banach space denoted by $(A_0, A_1)_{\theta,p}$.
\item Given $0<\theta<1$, the complex interpolation functor
constructs an intermediate Banach space denoted by $[A_0,
A_1]_{\theta}$.
\end{itemize}
\begin{theorem}[Real Interpolation]\lab{thm3.01}\cite{36}
 Let $\Omega$ be a bounded open set with smooth boundary in $\mathbb{R}^n$ or $\Omega=\mathbb{R}^n$. Suppose $\theta \in (0,1)$, $0\leq s_0,s_1<\infty$, and
 $1<p_0,p_1<\infty$. Additionally assume one of the following cases holds:
\begin{itemize}
\item $s_0, s_1, s$ are nonintegers.
\item  $s_0\in \mathbb{R}$, $s_1\in \mathbb{Z}$, and $s\in \mathbb{R}\setminus
\mathbb{Z}$.
\end{itemize}
 If
\begin{equation*}
s=(1-\theta)s_0+\theta s_1, \quad \quad
\frac{1}{p}=\frac{1-\theta}{p_0}+\frac{\theta}{p_1},
\end{equation*}
then $W^{s,p}(\Omega)=(W^{s_0,p_0}(\Omega),
W^{s_1,p_1}(\Omega))_{\theta,p}$.
\end{theorem}
\begin{theorem}[Complex Interpolation]\lab{thm3.02}\cite{36}
 Let $\Omega$ be a bounded open set with smooth boundary in $\mathbb{R}^n$ or $\Omega=\mathbb{R}^n$. Suppose $\theta \in (0,1)$,
 $0\leq s_0,s_1<\infty$, and
 $1<p_0,p_1<\infty$. If
\begin{equation*}
s=(1-\theta)s_0+\theta s_1, \quad \quad
\frac{1}{p}=\frac{1-\theta}{p_0}+\frac{\theta}{p_1},
\end{equation*}
then
\begin{itemize}
\item $H^{s,p}(\Omega)=[H^{s_0,p_0}(\Omega),
H^{s_1,p_1}(\Omega)]_{\theta}$.
\item $W^{s,p}(\Omega)=[W^{s_0,p_0}(\Omega),
W^{s_1,p_1}(\Omega)]_{\theta}$ provided $s_0, s_1, s>0$ are
nonintegers.
\item $W^{s,p}(\Omega)\hookrightarrow [W^{s_0,p_0}(\Omega),
W^{s_1,p_1}(\Omega)]_{\theta}$ provided $s_0$ and $s_1$ are not
integers and $p\geq 2$.\\
 (This is a consequence of the fact that
for $s_0,s_1\not \in \mathbb{Z}$, $[W^{s_0,p_0}(\Omega),
W^{s_1,p_1}(\Omega)]_{\theta}=B^{s}_{p,p}$. If $s\not \in
\mathbb{Z}$, then $B^{s}_{p,p}=W^{s,p}$; if $s\in \mathbb{Z}$,
then $W^{s,p}\hookrightarrow B^s_{p,p}$ provided $p\geq 2$.)
\end{itemize}
\end{theorem}
\begin{remark}
According to  \cite{36}, the above interpolation facts remain true
even if we only assume the bounded open set $\Omega$ is of
cone-type. According to  \cite{32} if $\Omega$ is a bounded open
set with Lipschitz continuous boundary, then it is of cone-type.
\end{remark}
\begin{theorem}[Properties of the Spaces $(A_0, A_1)_{\theta,p}$]\lab{thm8201}\cite{36}
\begin{itemize}
\item It holds that $(A_0, A_1)_{\theta,p}=(A_1, A_0)_{1-\theta,p}$.
\item If $A_0\hookrightarrow A_1$, then for $0<\theta<\tilde{\theta}<1$ and $1<p, \tilde{p}<\infty$
\begin{equation*}
(A_0, A_1)_{\theta,p}\hookrightarrow (A_0, A_1)_{\tilde{\theta},\tilde{p}}
\end{equation*}
\end{itemize}
\end{theorem}
\begin{theorem}[Interpolation Properties of Linear Operators]
\lab{thm3.03-B}\cite{37}
 Let $A_0\subseteq A_1$ and $B_0\subseteq B_1$ be couples of Banach spaces. If $T_1:A_1\rightarrow B_1$ is a continuous linear map that restricts to a continuous linear map $T_0: A_0\rightarrow B_0$, then $T_1$ also restricts to a continuous linear map from $(A_0, A_1)_{\theta,p}$ to $(B_0, B_1)_{\theta,p}$ for all $0<\theta<1$ and $1<p<\infty$.
\end{theorem}
\newpage
\begin{theorem}[Interpolation Properties of Bilinear Forms]
\lab{thm3.03}\cite{36}
Let $A_0 \subseteq A_1$, $B_0
\subseteq B_1$, and $C_0 \subseteq C_1$ be couples of Banach
spaces. If $T_1: A_1 \times B_1 \rightarrow C_1$ is a continuous
bilinear map that restricts to a continuous bilinear map $T_0:
A_0 \times B_0 \rightarrow C_0$, then $T_1$ also restricts to a
continuous bilinear map
\begin{itemize}
\item (complex interpolation) from $[A_0, A_1]_{\theta}\times [B_0, B_1]_{\theta}$ into $[C_0,
C_1]_{\theta}$, and
\item (real interpolation)  from $(A_0, A_1)_{\theta,p}\times (B_0,
B_1)_{\theta,q}$ into $(C_0, C_1)_{\theta,r}$
\end{itemize}
where $0<\theta<1$ and $\frac{1}{r}=\frac{1}{p}+\frac{1}{q}-1\geq
0$.
\end{theorem}
\begin{theorem}[Extension Property]\lab{thm3.1}\cite{33}
Let $\Omega \subset \mathbb{R}^n$ be a bounded open set with
Lipschitz continuous boundary. Then for all $s>0$ and for $1\leq
p < \infty$, there exists a continuous linear extension operator
$P: W^{s,p}(\Omega)\hookrightarrow W^{s,p}(\mathbb{R}^n)$ such
that $(Pu)|_{\Omega}=u$ and \newline $\| P
u\|_{W^{s,p}(\mathbb{R}^n)}\leq C\| u
\|_{W^{s,p}(\Omega)}$ for some constant $C$ that may
depend on $s$, $p$, and $\Omega$ but is independent of $u$.
\end{theorem}
\begin{theorem}[Embedding Theorem I]\lab{thm3.2}\cite{33,36,jZ77} Let $\Omega$ be a
bounded open subset of $\mathbb{R}^n$ with Lipschitz continuous
boundary or $\Omega=\mathbb{R}^n$. Suppose $1\leq p\leq q<\infty$
 and $0\leq t\leq s$
satisfy $s-\frac{n}{p}\geq t-\frac{n}{q}$.Then
\begin{itemize}
\item $W^{s,p}(\Omega)\hookrightarrow W^{t,q}(\Omega)$,
\item $H^{s,p}(\mathbb{R}^n)\hookrightarrow
H^{t,q}(\mathbb{R}^n)$ provided we assume $p>1$. 
\end{itemize}
\end{theorem}
\begin{theorem}[Embedding Theorem II]\cite{Gris85}\lab{thm3.3}
Let $\Omega$ be a bounded open subset of $\mathbb{R}^n$ with
Lipschitz continuous boundary or $\Omega=\mathbb{R}^n$.
\begin{enumerate}[(i)]
\item If $sp>n$, then $W^{s,p}(\Omega)\hookrightarrow
L^{\infty}(\Omega)\cap C^{0}(\Omega)$ and $W^{s,p}(\Omega)$ is a
Banach algebra.
\item If $sp=n$, then $W^{s,p}(\Omega)\hookrightarrow L^q(\Omega)$
for $p\leq q<\infty$.
\item If $0\leq sp<n$, then $W^{s,p}(\Omega)\hookrightarrow
L^{q}(\Omega)$ for $p\leq q \leq p^{*}=\frac{np}{n-sp}$.
\end{enumerate}
\end{theorem}
(Items (ii) and (iii) are direct consequences of Theorem
\ref{thm3.2}.)
The following result is a generalization of the well-known
embedding relationships; a simple proof does not appear to be in the
literature, so we include the short proof.
\begin{theorem}[Embedding Theorem III]\lab{thm3.4} Let $\Omega$ be a
bounded open subset of $\mathbb{R}^n$ with Lipschitz continuous
boundary. Suppose $1\leq p, q<\infty$ ($p$ does NOT need to be
less than $q$) and $0\leq t\leq s$ satisfy $s-\frac{n}{p}\geq
t-\frac{n}{q}$. If $s\not\in \mathbb{N}_0$, additionally assume that $s\neq t$. Then  $W^{s,p}(\Omega)\hookrightarrow
W^{t,q}(\Omega)$.
\end{theorem}
\begin{proof}{\bf (Theorem~\ref{thm3.4})}
 If $p\leq q$, the claim follows from Theorem
\ref{thm3.2}. So we may assume $p>q$. We consider three cases:
\begin{itemize}
\item \textbf{Case1 $s=t=k\in \mathbb{N}_0$:} Note that since $\Omega$
is a \textbf{bounded} open set, $L^{p}(\Omega)\hookrightarrow
L^{q}(\Omega)$. We can write
\begin{equation*}
\| u\|_{W^{k,q}(\Omega)} =\sum_{|\beta|\leq k}\|
\partial^{\beta}u\|_{L^{q}(\Omega)}
\preceq \sum_{|\beta|\leq k}\|
 \partial^{\beta}u\|_{L^{p}(\Omega)}
 =\|
u\|_{W^{k,p}(\Omega)}.
\end{equation*}
which precisely means that $W^{k,p}(\Omega)\hookrightarrow
W^{k,q}(\Omega)$.
\item \textbf{Case2 $\exists k \in \mathbb{N}_0$} such that \textbf{$t=k<s$:} It follows from Theorem
\ref{thm3.2} that $W^{s,p}(\Omega)\hookrightarrow W^{k,p}(\Omega)$. Now notice that, by what was proved in Case 1, $W^{k,p}(\Omega)\hookrightarrow
W^{k,q}(\Omega)$.
\item \textbf{Case3 $\exists k\in\mathbb{N}_0$} such that \textbf{$k<t<s<k+1$ (of course $p>q$):} Let
\begin{align*}
& \theta=s-\floor{s}\qquad (\textrm{so $s=(1-\theta)k+\theta(k+1)$})\\
& \tilde{\theta}=t-\floor{t}\qquad (\textrm{so $t=(1-\tilde{\theta})k+\tilde{\theta}(k+1)$})
\end{align*}
Note that since $t<s$, we have $\tilde{\theta}<\theta$. By what was show in Case 1:
\begin{equation*}
W^{k,p}(\Omega)\hookrightarrow W^{k,q}(\Omega),\qquad W^{k+1,p}(\Omega)\hookrightarrow W^{k+1,q}(\Omega)
\end{equation*}
Since $s=(1-\theta)k+\theta(k+1)$, it follows from interpolation properties of linear operators (Theorem \ref{thm3.03-B}) that
\begin{align*}
W^{s,p}(\Omega)&=(W^{k,p}(\Omega),W^{k+1,p}(\Omega))_{\theta,p}\hookrightarrow (W^{k,q}(\Omega), W^{k+1,q}(\Omega))_{\theta,p}\\
&\stackrel{\textrm{Theorem \ref{thm8201}}}{=}(W^{k+1,q}(\Omega), W^{k,q}(\Omega))_{1-\theta,p}
\end{align*}
Since $1-\theta<1-\tilde{\theta}$ and $W^{k+1,q}\hookrightarrow W^{k,q}$, it follows from Theorem \ref{thm8201} that
\begin{align*}
(W^{k+1,q}(\Omega), W^{k,q}(\Omega))_{1-\theta,p}&\hookrightarrow (W^{k+1,q}(\Omega), W^{k,q}(\Omega))_{1-\tilde{\theta},q}\\
&\hookrightarrow
(W^{k,q}(\Omega), W^{k+1,q}(\Omega))_{\tilde{\theta},q}=W^{t,q}(\Omega)\,.
\end{align*}
Thus $W^{s,p}(\Omega)\hookrightarrow W^{t,q}(\Omega)$ as desired.
\item \textbf{Case4 $\exists k\in\mathbb{N}$} such that \textbf{$t<k<s$ (of course $p>q$):}
Let $\hat{s}$ be a number in the open interval $(t,\floor{t}+1)$. It follows from Theorem
\ref{thm3.2} that $W^{s,p}(\Omega)\hookrightarrow W^{\hat{s},p}(\Omega)$. Now notice that, by what was proved in Case 3,  $W^{\hat{s},p}(\Omega)\hookrightarrow
W^{t,q}(\Omega)$.
\end{itemize}
\end{proof}
\section{A Counter-Example for Generalized Holder-Type Inequalities in $W^{s,p}$}
\label{subsec:maintheorems} Before stating
the main theorems, we discuss a simple case which demonstrates
that multiplication properties of Sobolev-Slobodeckij spaces can
be quite counterintuitive.

\textbf{Notation:} Let $A_i$ and $B_i$ ($i=1,2$) and $C$ be
Sobolev spaces.
\begin{itemize}
\item By writing $A_1 \times A_2 \hookrightarrow B_1 \times B_2$ we
merely mean
 that $A_1 \times A_2 \subseteq B_1 \times B_2$ and if $u\in
A_1$ and $v\in A_2$, then $\|u\|_{B_1}\|v\|_{B_2}\preceq
\|u\|_{A_1}\|v\|_{A_2}$. $(A_1\times A_2=\{a_1a_2:a_1\in
A_1\,,a_2\in A_2\})$
\item By writing $B_1\times B_2\hookrightarrow C$ we mean that $B_1\times B_2\subseteq
C$ and if $u\in B_1$ and $v\in B_2$, then $\|uv\|_C\preceq
\|u\|_{B_1}\|v\|_{B_2}$.
\end{itemize}
\begin{theorem}\lab{thmsobolevholdercount}
Suppose $k\in \mathbb{N}_0$, and
$\frac{1}{p_1}+\frac{1}{p_2}=\frac{1}{p}$. Then
\begin{equation*}
W^{k,p_1}(\mathbb{R}^n)\times
W^{k,p_2}(\mathbb{R}^n)\hookrightarrow W^{k,p}(\mathbb{R}^n).
\end{equation*}
More generally, if $s\geq 0$, then
\begin{equation*}
H^{s,p_1}(\mathbb{R}^n)\times
H^{s,p_2}(\mathbb{R}^n)\hookrightarrow H^{s,p}(\mathbb{R}^n).
\end{equation*}
\end{theorem}
\begin{proof}{\bf (Theorem~\ref{thmsobolevholdercount})}
 For $k\in \mathbb{N}_0$ the claim is a direct consequence of
the definition of Sobolev norm, Leibniz formula
$(\partial^{\alpha}(uv)=\sum_{\beta\leq \alpha} {\alpha \choose
\beta}
\partial^{\alpha-\beta}u\partial^{\beta}v)$, and the Holder's
inequality for Lebesgue spaces.

If $s\not \in \mathbb{N}_0$, then let $k=\floor{s}$ and
$\theta=s-k$. We have
\begin{align*}
& H^{k,p_1}(\mathbb{R}^n)\times
H^{k,p_2}(\mathbb{R}^n)\hookrightarrow H^{k,p}(\mathbb{R}^n),\\
& H^{k+1,p_1}(\mathbb{R}^n)\times
H^{k+1,p_2}(\mathbb{R}^n)\hookrightarrow H^{k+1,p}(\mathbb{R}^n).
\end{align*}
Since
\begin{equation*}
H^{s,p}=[H^{k,p}, H^{k+1,p}]_{\theta},\quad H^{s,p_1}=[H^{k,p_1},
H^{k+1,p_1}]_{\theta},\quad H^{s,p_2}=[H^{k,p_2},
H^{k+1,p_2}]_{\theta}
\end{equation*}
the claim follows from complex interpolation.
\end{proof}
 Now we ask the following question: does the claim of Theorem
\ref{thmsobolevholdercount} hold true for Sobolev-Slobodeckij
spaces? More specifically, suppose $s>0$, $s\not\in \mathbb{Z}$,
and $\frac{1}{p_1}+\frac{1}{p_2}=\frac{1}{p}$. Can we conclude
that $W^{s,p_1}(\mathbb{R}^n)\times
W^{s,p_2}(\mathbb{R}^n)\hookrightarrow W^{s,p}(\mathbb{R}^n)$?
Surprisingly, the answer is \textbf{NO}! In what follows we will
specialize the argument given in \cite{37} for Triebel-Lizorkin
spaces to the case of Sobolev-Slobodeckij spaces to show that if
$s\not \in \mathbb{Z}$ (and of course $s>0$) for
$W^{s,p_1}(\mathbb{R}^n)\times
W^{s_2,p_2}(\mathbb{R}^n)\hookrightarrow W^{s,p}(\mathbb{R}^n)$
to be true it is necessary to have $p_1\leq p$.
\begin{lemma}
Suppose $s>0$ is given. Let $f\in S(\reals^n)$ be a function such
that
\begin{equation*}
{\rm supp} \mathcal{F}f \subseteq \{\xi : |\xi|<\epsilon
\},\,\quad f\not \equiv 0.
\end{equation*}
If $\epsilon$ is sufficiently small, then there exists a sequence
of functions $\{g_N\}_{N=1}^{\infty}$ (each $g_N$ depends on $s$)
such that for any $p,q>1$
\begin{equation*}
\|g_N\|_{F^s_{p,q}}=N^{\frac{1}{q}}\|f\|_p \quad
\textrm{and}\quad \|g_N f\|_{F^s_{p,q}}=N^{\frac{1}{q}}\|f^2\|_p.
\end{equation*}
\end{lemma}
 The construction of $g_N$'s is based on the Littlewood-Paley
characterization of Triebel- Lizorkin spaces and can be found in
\cite{37}.
\begin{proposition}\lab{propholderwrong}
Suppose $s, s_2\geq 0$, $s\not \in \mathbb{Z}$ and $p_1,p_2,p>
1$. If $W^{s,p_1}(\mathbb{R}^n)\times
W^{s_2,p_2}(\mathbb{R}^n)\hookrightarrow W^{s,p}(\mathbb{R}^n)$,
then $p_1\leq p$.
\end{proposition}
\begin{proof}{\bf (Proposition~\ref{propholderwrong})}
 Note that, since $s\not \in \mathbb{Z}$, we have
$W^{s,p}=F^s_{p,p}$. Consider the product of $f$ and $g_N $; by
assumption we must have
\begin{equation*}
\|g_N. f\|_{W^{s,p}}\preceq \|g_N\|_{W^{s,p_1}}\|f\|_{W^{s_2,p_2}}
\end{equation*}
where the implicit constant is independent of $N$. Therefore
\begin{equation*}
N^{\frac{1}{p}}\|f^2\|_p\preceq
 N^{\frac{1}{p_1}}\|f\|_{p_1}\|f\|_{W^{s_2,p_2}}.
\end{equation*}
So for all $N\in \mathbb{N}$
\begin{equation*}
 0<
\frac{\|f^2\|_p}{\|f\|_{p_1}\|f\|_{W^{s_2,p_2}}}\preceq
N^{\frac{1}{p_1}-\frac{1}{p}},
\end{equation*}
Which implies that $p_1\leq p$.
\end{proof}
\begin{remark}
Proposition~\ref{propholderwrong}, in part, shows that the claim
of Theorem 1.4.4.2 of \cite{Gris85} (in the generality that is
stated in \cite{Gris85}) does not hold true. Also the claim
stated in part (d) of page 47 in \cite{Amann91} (in the generality
that is stated in \cite{Amann91}) does not hold true.
\end{remark}

\section{Sufficient Conditions for $H^{s_1,p_1}\times H^{s_2,p_2}\hookrightarrow
H^{s,p}$,
         $s\geq 0$, $s \in \mathbb{R}$}
\label{sec:bessel}
 We start our main theorems by a theorem on multiplication in
spaces $H^{s,p}(\mathbb{R}^n)$ with $s\geq 0$. The reason that we
begin with a theorem on Bessel potential spaces is that although
for these spaces the situation is considerably simpler (comparing
to Sobolev-Slobodeckij spaces), it showcases the main ideas
 without encountering technical difficulties. The aforementioned
simplicity is due to the fact that we have a uniform formula for
the space
$[H^{s_0,p_0}, H^{s_1,p_1}]_{\theta}$ regardless of whether each
of $s_0$, $s_1$, or $(1-\theta)s_0+\theta s_1$ is an integer or
not.
This first result is classical and well-known; however, the
following fairly short proof based on complex interpolation and
embedding theorems does not appear to be in the literature, so we
include it.
\begin{theorem}[Pointwise multiplication in spaces $H^{s,p}(\mathbb{R}^n)$ with $s\geq 0$]\lab{thm4.0}
Assume $s_i,s$ and $1 < p_i \leq p< \infty$ ($i=1,2$) are real
numbers satisfying
\begin{enumerate}[(i)]
\item  $s_i \geq s$, 
\item  $s\geq 0$,
\item  $s_i-s\geq n(\dfrac{1}{p_i}-\dfrac{1}{p})$,
\item  $s_1+s_2-s>n(\dfrac{1}{p_1}+\dfrac{1}{p_2}-\dfrac{1}{p})$.
\end{enumerate}
Claim: If $u\in H^{s_1,p_1}(\mathbb{R}^n)$ and $v\in
H^{s_2,p_2}(\mathbb{R}^n)$, then $uv \in H^{s,p}(\mathbb{R}^n)$,
and moreover, the pointwise multiplication of functions is a
continuous bilinear map
\begin{equation*}
H^{s_1,p_1}(\mathbb{R}^n)\times
H^{s_2,p_2}(\mathbb{R}^n)\rightarrow H^{s,p}(\mathbb{R}^n).
\end{equation*}
\end{theorem}
\begin{proof}{\bf (Theorem~\ref{thm4.0})}
 Our proof consists of two steps. In the first step
we consider the special case $p_1=p_2=p$, and then in the second
step we prove the general case based on the special case that is
proved in Step 1.
\begin{itemizeXALI}
\item \textbf{Step 1:} Here we want to prove the theorem for the
special case $p=p_1=p_2$. In this case the assumptions can be
rewritten as follows:
\begin{equation*}
s_1,s_2\geq s \geq 0, \quad s_1+s_2-s>\frac{n}{p}.
\end{equation*}
In order to proceed, we state and prove a simple lemma.
\begin{lemma}\lab{lem4.1}
\begin{align*}
& \forall\, \epsilon>0 \quad \forall t\in[0,\frac{n}{p}]\quad
H^{t,p}(\mathbb{R}^n)\times H^{\frac{n}{p}+\epsilon,p}(\mathbb{R}^n)\hookrightarrow H^{t,p}(\mathbb{R}^n).\\
& \forall\, \epsilon>0 \quad \forall t\in[0,\frac{n}{p}]\quad
H^{\frac{n}{p}+\epsilon,p}(\mathbb{R}^n) \times
H^{t,p}(\mathbb{R}^n)\hookrightarrow H^{t,p}(\mathbb{R}^n).
\end{align*}
\end{lemma}
\textbf{Proof of the Lemma} Clearly it is enough to prove the
first statement. Let $\epsilon>0$ be given. Since
$\frac{n}{p}+\epsilon>\frac{n}{p}$,
$H^{\frac{n}{p}+\epsilon,p}(\mathbb{R}^n)$ is an algebra and
\begin{equation}\lab{eqnbesselalgebra11}
H^{\frac{n}{p}+\epsilon,p}(\mathbb{R}^n)\times
H^{\frac{n}{p}+\epsilon,p}(\mathbb{R}^n)\hookrightarrow
H^{\frac{n}{p}+\epsilon,p}(\mathbb{R}^n).
\end{equation}
Also $H^{\frac{n}{p}+\epsilon,p}(\mathbb{R}^n)\hookrightarrow
L^{\infty}(\mathbb{R}^n)$. Hence
\begin{equation}\lab{eqnlebesgue1}
H^{\frac{n}{p}+\epsilon,p}(\mathbb{R}^n)\times
H^{0,p}(\mathbb{R}^n)\hookrightarrow H^{0,p}(\mathbb{R}^n)\quad
(L^{\infty}\times L^{p}\hookrightarrow L^{p})\,.
\end{equation}
By complex interpolation between (\ref{eqnbesselalgebra11}) and
(\ref{eqnlebesgue1}) we get
\begin{equation*}
\forall\, \theta \in [0,1] \quad
H^{\frac{n}{p}+\epsilon,p}(\mathbb{R}^n)\times
H^{\theta(\frac{n}{p}+\epsilon),p}(\mathbb{R}^n)\hookrightarrow
H^{\theta(\frac{n}{p}+\epsilon),p}(\mathbb{R}^n)
\end{equation*}
which clearly implies the claim of the Lemma.

Now, using the above lemma, we can prove the theorem for the
special case $p=p_1=p_2$. To this end we consider two cases:
\begin{itemizeXX}
\item \textbf{Case 1 $s>\frac{n}{p}$:} If $s>\frac{n}{p}$, then
 $H^{s,p}(\mathbb{R}^n)$is an algebra and we can write
\begin{align*}
H^{s_1,p_1}(\mathbb{R}^n) \times
H^{s_2,p_2}(\mathbb{R}^n)&\hookrightarrow
H^{s,p}(\mathbb{R}^n)\times
H^{s,p}(\mathbb{R}^n)\quad (\textrm{by assumption $s_1,s_2\geq s$})\\
&\hookrightarrow H^{s,p}(\mathbb{R}^n).
\end{align*}
\item \textbf{Case 2 $s\leq \frac{n}{p}$:} Let
$\epsilon=s_1+s_2-s-\frac{n}{p}>0$. By Lemma \ref{lem4.1} we have
\begin{align}\lab{complexintembed1}
&
H^{s,p}(\mathbb{R}^n)\times H^{\frac{n}{p}+\epsilon,p}(\mathbb{R}^n)\hookrightarrow H^{s,p}(\mathbb{R}^n).\\
& H^{\frac{n}{p}+\epsilon,p}(\mathbb{R}^n) \times
H^{s,p}(\mathbb{R}^n)\hookrightarrow
H^{s,p}(\mathbb{R}^n).\lab{complexintembed2}
\end{align}
Note that
\begin{equation*}
s\leq s_2 \Longrightarrow s_1\leq s_1+s_2-s \Longrightarrow s_1
\leq \frac{n}{p}+\epsilon.
\end{equation*}
So there exists $\theta \in [0,1]$ such that
$(1-\theta)s+\theta(\frac{n}{p}+\epsilon)=s_1$. Clearly
\begin{equation*}
[(1-\theta)s+\theta(\frac{n}{p}+\epsilon)]+[(1-\theta)(\frac{n}{p}+\epsilon)+\theta
 s]=s+\frac{n}{p}+\epsilon=s_1+s_2.
\end{equation*}
That is, $s_1+ [(1-\theta)(\frac{n}{p}+\epsilon)+\theta s
]=s_1+s_2$ which means that
 $(1-\theta)(\frac{n}{p}+\epsilon)+\theta s=s_2$. Consequently
\begin{equation*}
[H^{s,p}(\mathbb{R}^n),H^{\frac{n}{p}+\epsilon,p}(\mathbb{R}^n)]_{\theta}=H^{s_1,p}(\mathbb{R}^n),\quad
[H^{\frac{n}{p}+\epsilon,p}(\mathbb{R}^n),
H^{s,p}(\mathbb{R}^n)]_{\theta}=H^{s_2,p}(\mathbb{R}^n).
\end{equation*}
So using complex interpolation, (\ref{complexintembed1}), and
(\ref{complexintembed2}) we get
\begin{equation*}
H^{s_1,p_1}(\mathbb{R}^n)\times
H^{s_2,p_2}(\mathbb{R}^n)\hookrightarrow H^{s,p}(\mathbb{R}^n).
\end{equation*}
\end{itemizeXX}
\item \textbf{Step 2:} Now we are in the position to prove the
general case. Let
\begin{equation*}
\tilde{s}_1=s_1-\frac{n}{p_1}+\frac{n}{p},\quad
\tilde{s}_2=s_2-\frac{n}{p_2}+\frac{n}{p}.
\end{equation*}
We just need to prove the following claim:\\
\noindent \textbf{Claim:}
\begin{enumerate}[(i)]
\item $H^{\tilde{s}_1,p}(\mathbb{R}^n)\times H^{\tilde{s}_2,p}(\mathbb{R}^n)\hookrightarrow
 H^{s,p}(\mathbb{R}^n)$.
\item $H^{s_1,p_1}(\mathbb{R}^n)\hookrightarrow H^{\tilde{s}_1,p}(\mathbb{R}^n)$.
\item $H^{s_2,p_2}(\mathbb{R}^n)\hookrightarrow H^{\tilde{s}_2,p}(\mathbb{R}^n)$.
\end{enumerate}
Indeed, if we prove the above claim, then
\begin{equation*}
H^{s_1,p_1}(\mathbb{R}^n)\times
H^{s_2,p_2}(\mathbb{R}^n)\hookrightarrow
H^{\tilde{s}_1,p}(\mathbb{R}^n)\times
H^{\tilde{s}_2,p}(\mathbb{R}^n)\hookrightarrow
H^{s,p}(\mathbb{R}^n).
\end{equation*}
\begin{itemizeXX}
\item \textbf{Proof of (i):} By Step 1 we need to check the
following items:
\begin{align*}
& \tilde{s}_1\geq s \quad (\textrm{true because $s_1-s\geq
n(\frac{1}{p_1}-\frac{1}{p})$})\\
& \tilde{s}_2\geq s \quad (\textrm{true because $s_2-s\geq
n(\frac{1}{p_2}-\frac{1}{p})$})\\
& \tilde{s}_1+\tilde{s}_2-s>\frac{n}{p}
\end{align*}
The last item is true because
\begin{equation*}
s_1+s_2-s>n(\frac{1}{p_1}+\frac{1}{p_2}-\frac{1}{p})\Longrightarrow
(s_1-\frac{n}{p_1}+\frac{n}{p})+(s_2-\frac{n}{p_2}+\frac{n}{p})-s>\frac{n}{p}\Longrightarrow
\tilde{s}_1+\tilde{s}_2-s>\frac{n}{p}.
\end{equation*}
\item \textbf{Proof of (ii):} According to Embedding Theorem I we
must check the following items:
\begin{align*}
&p_1\leq p \quad (\textrm{true by assumption})\\
& s_1\geq \tilde{s}_1 \quad (\textrm{true because $p\geq p_1
\Rightarrow \frac{n}{p_1}\geq \frac{n}{p}\Rightarrow s_1\geq
s_1-\frac{n}{p_1}+\frac{n}{p}$})\\
& s_1-\frac{n}{p_1}\geq \tilde{s}_1-\frac{n}{p} \quad
(\textrm{true because
$\tilde{s}_1-\frac{n}{p}=s_1-\frac{n}{p_1}+\frac{n}{p}-\frac{n}{p}=s_1-\frac{n}{p_1}$
})
\end{align*}
\item \textbf{Proof of (iii):} Completely analogous to the proof of the
previous item!
\end{itemizeXX}
\end{itemizeXALI}
\end{proof}
\newpage
\section{Sufficient Conditions for $W^{s_1,p_1}\times W^{s_2,p_2}\hookrightarrow
W^{s,p}$,
         $s\geq 0$, $s \in \mathbb{N}_0$}
\label{sec:integers}

We now consider the case where the product belongs to a Sobolev
space with integer smoothness index. The proof of the following
theorem is based on the classical definition of Sobolev spaces,
Holder's inequality for Lebesgue spaces, and previously stated
embedding theorems.

\begin{theorem}\lab{thm4.6}
Let $s_i, s$ and $1 \leq p, p_i < \infty$ ($i=1,2$) be real
numbers satisfying
\begin{enumerate}[(i)]
\item  $s_i \geq s\geq 0$ 
\item  $s\in \mathbb{N}_0$,
\item  $s_i-s\geq n(\dfrac{1}{p_i}-\dfrac{1}{p})$,
\item  $s_1+s_2-s>n(\dfrac{1}{p_1}+\dfrac{1}{p_2}-\dfrac{1}{p})\geq
0$.
\end{enumerate}
where the strictness of the inequalities in items (iii) and (iv)
can be interchanged.\\ Claim: If $u\in W^{s_1,p_1}(\mathbb{R}^n)$
and $v\in W^{s_2,p_2}(\mathbb{R}^n)$, then $uv \in
W^{s,p}(\mathbb{R}^n)$ and moreover the pointwise multiplication
of functions is a continuous bilinear map
\begin{equation*}
W^{s_1,p_1}(\mathbb{R}^n)\times
W^{s_2,p_2}(\mathbb{R}^n)\rightarrow W^{s,p}(\mathbb{R}^n).
\end{equation*}
\end{theorem}
\begin{remark}
\label{rem:1}
Note that $p_i$ is not required to be
less than or equal to $p$ in the statement of Theorem~\ref{thm4.6}.
It is the restriction that $s$ be an integer in the theorem that
makes it possible to remove the ordering between $p_i$ and $p$.
We will see below in Theorem~\ref{thm4.2} that alternatively, one can
restrict consideration to a bounded domain $\Omega$ in place
of $\mathbb{R}^n$, allowing $s$ to be noninteger, yet still removing
the ordering restriction between $p_i$ and $p$.
\end{remark}
\begin{proof}{\bf (Theorem~\ref{thm4.6})}
Let $u\in W^{s_1,p_1}(\mathbb{R}^n)$ and $v\in
W^{s_2,p_2}(\mathbb{R}^n)$. Our goal is to prove that $\|
 uv\|_{s,p}\preceq \| u\|_{s_1,p_1}\| v\|_{s_2,p_2}$. We have
\begin{equation*}
\| uv\|_{s,p}=\sum_{|\alpha|\leq s}\|
 \partial^{\alpha}(uv)\|_{p}.
\end{equation*}
So, it is enough to prove that for all $|\alpha|\leq s$, $\|
 \partial^{\alpha}(uv)\|_{p}\preceq \| u\|_{s_1,p_1}\|
 v\|_{s_2,p_2}$. For now let's assume $v\in
 C_c^{\infty}(\mathbb{R}^n)$. So we can use the Leibniz
 formula  (see e.g. \cite{32}) to write
\begin{equation*}
\partial^{\alpha}(uv)=\sum_{\beta\leq \alpha} {\alpha \choose
\beta}
\partial^{\alpha-\beta}u\partial^{\beta}v.
\end{equation*}
Thus we just need to show that
\begin{equation*}
\forall\, |\alpha|\leq s\quad \forall \, \beta\leq \alpha \quad
\| \partial^{\alpha-\beta}u\partial^{\beta}v
\|_p \preceq \|
u\|_{s_1,p_1}
\| v\|_{s_2,p_2}.
\end{equation*}
Fix $\alpha,\beta \in \mathbb{N}_0^n$ such that $|\alpha|\leq s$
and $\beta\leq \alpha$. In what follows we will
prove the following claim:\\
\noindent \textbf{Claim:} There exist $r\in [1,\infty]$ and $q\in
[1,\infty]$ such that
\begin{equation}\lab{eqnholderembed1}
\frac{1}{r}+\frac{1}{q}=\frac{1}{p},\quad
W^{s_1-|\alpha-\beta|,p_1}(\mathbb{R}^n)\hookrightarrow
L^r(\mathbb{R}^n), \quad
W^{s_2-|\beta|,p_2}(\mathbb{R}^n)\hookrightarrow
L^q(\mathbb{R}^n).
\end{equation}
For the moment, let's assume the above claim is true. Then
\begin{align*}
& u\in W^{s_1,p_1}(\mathbb{R}^n) \Longrightarrow \partial
^{\alpha-\beta} u \in
W^{s_1-|\alpha-\beta|,p_1}(\mathbb{R}^n)\hookrightarrow
L^r(\mathbb{R}^n),\\
& v\in W^{s_2,p_2}(\mathbb{R}^n) \Longrightarrow \partial
^{\beta} v \in W^{s_2-|\beta|,p_2}(\mathbb{R}^n)\hookrightarrow
L^q(\mathbb{R}^n),
\end{align*}
and therefore
\begin{align*}
\| \partial^{\alpha-\beta}u\partial^{\beta}v
\|_p \leq \| \partial^{\alpha-\beta}u\|_{r}
\| \partial^{\beta}v
\|_q &\preceq
 \| \partial^{\alpha-\beta}u\|_{s_1-|\alpha-\beta|,p_1}
\| \partial^{\beta}v
\|_{s_2-|\beta|,p_2}\\
&\preceq \| u\|_{s_1,p_1}
\| v\|_{s_2,p_2}\,.
\end{align*}
So it is enough to prove the above claim. We consider two cases
separately:\\
\textbf{Case 1:} $s_i-s> n(\frac{1}{p_i}-\frac{1}{p})$ $(i=1,2)$
and $s_1+s_2-s \geq n(\frac{1}{p_1}+\frac{1}{p_2}-\frac{1}{p})\geq 0$.

As a direct consequence of assumptions we have
\begin{align}
& \frac{1}{p_1}-\frac{s_1-|\alpha-\beta|}{n}\leq
\frac{1}{p_1}-\frac{s_1-s}{n}<\frac{1}{p}.\lab{eqngood1}\\
& \frac{1}{p}-\frac{1}{p_2}+\frac{s_2-|\beta|}{n}\geq
\frac{1}{p}-\frac{1}{p_2}+\frac{s_2-s}{n}>0. \lab{eqngood2}
\end{align}
In what follows we will show that there exist $r\in [1,\infty)$
and $q\in [1,\infty)$ that satisfy (\ref{eqnholderembed1}).
According to Theorem \ref{thm3.2} it is enough to show that there
exist $r$ and $q$ that satisfy the following conditions:
\begin{align*}
&0<\frac{1}{r}\leq 1,\quad  0<\frac{1}{q}\leq 1\, \quad
\frac{1}{r}+\frac{1}{q}=\frac{1}{p},\\
&\frac{1}{r}\leq\frac{1}{p_1},\quad \frac{1}{q}\leq \frac{1}{p_2},\\
& s_1-|\alpha-\beta|-\frac{n}{p_1}\geq 0-\frac{n}{r},\quad
s_2-|\beta|-\frac{n}{p_2}\geq 0-\frac{n}{q}.
\end{align*}
In fact, if we let $R=\frac{1}{r}$ and $Q=\frac{1}{q}$, then our
goal is to show that there exist $0<R\leq 1$ and $0<Q\leq 1$ such
that
\begin{align*}
R+Q= \frac{1}{p},\quad
\frac{1}{p_1}-\frac{s_1-|\alpha-\beta|}{n}\leq R\leq
\frac{1}{p_1},\quad \frac{1}{p_2}-\frac{s_2-|\beta|}{n}\leq Q
\leq \frac{1}{p_2}.
\end{align*}
Note that since $\frac{1}{p_1}\leq 1$ and $\frac{1}{p_2}\leq 1$,
conditions $R\leq 1$ and $Q\leq 1$ are superfluous. So, we need to
show that there exists $0<R<\frac{1}{p}$ such that
\begin{align*}
& \frac{1}{p_1}-\frac{s_1-|\alpha-\beta|}{n}\leq R\leq
\frac{1}{p_1},\\
& \frac{1}{p_2}-\frac{s_2-|\beta|}{n}\leq \frac{1}{p}-R\leq
\frac{1}{p_2}\quad (\Longleftrightarrow
\frac{1}{p}-\frac{1}{p_2}\leq R\leq
\frac{1}{p}-\frac{1}{p_2}+\frac{s_2-|\beta|}{n}).
\end{align*}
Consequently, it is enough to show that the following intersection
is nonempty:
\begin{equation*}
(0,\frac{1}{p})\cap
[\frac{1}{p_1}-\frac{s_1-|\alpha-\beta|}{n},\frac{1}{p_1}]\cap
[\frac{1}{p}-\frac{1}{p_2},\frac{1}{p}-\frac{1}{p_2}+\frac{s_2-|\beta|}{n}].
\end{equation*}
Note that by (\ref{eqngood1}),
$\frac{1}{p_1}-\frac{s_1-|\alpha-\beta|}{n}<\frac{1}{p}$ and so
the first intersection is nonempty. We may consider four cases:
\begin{enumerate}[(i)]
\item \textbf{$\frac{1}{p_1}-\frac{s_1-|\alpha-\beta|}{n}\leq 0$,
$\frac{1}{p_1}<\frac{1}{p}$:}
\begin{equation*}
(0,\frac{1}{p})\cap
[\frac{1}{p_1}-\frac{s_1-|\alpha-\beta|}{n},\frac{1}{p_1}]=(0,\frac{1}{p_1}].
\end{equation*}
Now note that by assumption $\frac{1}{p_1}\geq
\frac{1}{p}-\frac{1}{p_2}$ and also by (\ref{eqngood2}),
$\frac{1}{p}-\frac{1}{p_2}+\frac{s_2-|\beta|}{n}>0$. Hence
\begin{equation*}
(0,\frac{1}{p_1}] \cap
[\frac{1}{p}-\frac{1}{p_2},\frac{1}{p}-\frac{1}{p_2}+\frac{s_2-|\beta|}{n}]\neq
\emptyset
\end{equation*}
\item \textbf{$\frac{1}{p_1}-\frac{s_1-|\alpha-\beta|}{n}\leq 0$,
$\frac{1}{p_1}\geq \frac{1}{p}$:}
\begin{equation*}
(0,\frac{1}{p})\cap
[\frac{1}{p_1}-\frac{s_1-|\alpha-\beta|}{n},\frac{1}{p_1}]=(0,\frac{1}{p}).
\end{equation*}
Clearly $\frac{1}{p}> \frac{1}{p}-\frac{1}{p_2}$ and also by
(\ref{eqngood2}),
$\frac{1}{p}-\frac{1}{p_2}+\frac{s_2-|\beta|}{n}>0$. Hence
\begin{equation*}
(0,\frac{1}{p}) \cap
[\frac{1}{p}-\frac{1}{p_2},\frac{1}{p}-\frac{1}{p_2}+\frac{s_2-|\beta|}{n}]\neq
\emptyset
\end{equation*}
\item \textbf{$\frac{1}{p_1}-\frac{s_1-|\alpha-\beta|}{n}> 0$,
$\frac{1}{p}\leq \frac{1}{p_1}$:}
\begin{equation*}
(0,\frac{1}{p})\cap
[\frac{1}{p_1}-\frac{s_1-|\alpha-\beta|}{n},\frac{1}{p_1}]=[\frac{1}{p_1}-\frac{s_1-|\alpha-\beta|}{n},\frac{1}{p}).
\end{equation*}
Clearly $\frac{1}{p}> \frac{1}{p}-\frac{1}{p_2}$ and also by
assumption $s_1+s_2-s\geq
n(\frac{1}{p_1}+\frac{1}{p_2}-\frac{1}{p})$ and so
$\frac{1}{p_1}-\frac{s_1-|\alpha-\beta|}{n}\leq
\frac{1}{p}-\frac{1}{p_2}+\frac{s_2-|\beta|}{n}$. Consequently
\begin{equation*}
[\frac{1}{p_1}-\frac{s_1-|\alpha-\beta|}{n},\frac{1}{p})\cap
[\frac{1}{p}-\frac{1}{p_2},\frac{1}{p}-\frac{1}{p_2}+\frac{s_2-|\beta|}{n}]\neq
\emptyset
\end{equation*}
\item \textbf{$\frac{1}{p_1}-\frac{s_1-|\alpha-\beta|}{n}> 0$,
$\frac{1}{p_1}< \frac{1}{p}$:}
\begin{equation*}
(0,\frac{1}{p})\cap
[\frac{1}{p_1}-\frac{s_1-|\alpha-\beta|}{n},\frac{1}{p_1}]=[\frac{1}{p_1}-\frac{s_1-|\alpha-\beta|}{n},\frac{1}{p_1}].
\end{equation*}
By assumption $\frac{1}{p_1}\geq \frac{1}{p}-\frac{1}{p_2}$ and
also (exactly the same as the previous item)
$\frac{1}{p_1}-\frac{s_1-|\alpha-\beta|}{n}\leq
\frac{1}{p}-\frac{1}{p_2}+\frac{s_2-|\beta|}{n}$. Consequently
\begin{equation*}
[\frac{1}{p_1}-\frac{s_1-|\alpha-\beta|}{n},\frac{1}{p_1}]\cap
[\frac{1}{p}-\frac{1}{p_2},\frac{1}{p}-\frac{1}{p_2}+\frac{s_2-|\beta|}{n}]\neq
\emptyset
\end{equation*}
\end{enumerate}
\textbf{Case 2:} $s_i-s\geq n(\frac{1}{p_i}-\frac{1}{p})$
$(i=1,2)$ and $s_1+s_2-s >
n(\frac{1}{p_1}+\frac{1}{p_2}-\frac{1}{p})\geq 0$.

If $s_i-s>n(\frac{1}{p_i}-\frac{1}{p})$ ($i=1,2$), then the proof
of previous case works. So we just need to consider the following
cases:
\begin{enumerate}[(i)]
\item \textbf{$s_1-s=n(\frac{1}{p_1}-\frac{1}{p})$, $s_2-s\neq n(\frac{1}{p_2}-\frac{1}{p})$:} If $|\alpha-\beta|<s$, then the proof of \textbf{Case 1} works. In fact,
note that the proof of \textbf{Case 1} was based on the
inequalities
$\frac{1}{p_1}-\frac{s_1-|\alpha-\beta|}{n}<\frac{1}{p}$ and $
\frac{1}{p}-\frac{1}{p_2}+\frac{s_2-|\beta|}{n}>0$
((\ref{eqngood1}) and (\ref{eqngood2})) and both inequalities hold
true in this case: the second inequality is true because as in
\textbf{Case 1} $s_2-s> n(\frac{1}{p_2}-\frac{1}{p})$, and the
first inequality is true because
\begin{align*}
& \frac{1}{p_1}-\frac{s_1-|\alpha-\beta|}{n}<
\frac{1}{p_1}-\frac{s_1-s}{n}\leq\frac{1}{p}.
\end{align*}
So we may assume $|\alpha-\beta|=s$. Since $|\alpha|\leq s$ and
$\beta\leq \alpha$, this is possible only if
  $|\alpha|=s$ and $|\beta|=0$.

By assumption $s_1+s_2-s>
n(\frac{1}{p_1}+\frac{1}{p_2}-\frac{1}{p})$, so
$s_2>\frac{n}{p_2}$. Also $s_1-s\geq 0$ and therefore $p_1\leq p$.
Consequently
\begin{equation*}
W^{s_1-s,p_1}(\mathbb{R}^n)\hookrightarrow L^p(\mathbb{R}^n),\quad
W^{s_2,p_2}(\mathbb{R}^n)\hookrightarrow L^{\infty}(\mathbb{R}^n).
\end{equation*}
That is, (\ref{eqnholderembed1}) is satisfied with $r=p$ and
$q=\infty$. (Note that $|\alpha-\beta|=s$ and $|\beta|=0$)
\item \textbf{$s_2-s=n(\frac{1}{p_2}-\frac{1}{p})$, $s_1-s\neq n(\frac{1}{p_1}-\frac{1}{p})$:} If $|\beta|<s$, then the proof of \textbf{Case 1} works (again because inequalities
$\frac{1}{p_1}-\frac{s_1-|\alpha-\beta|}{n}<\frac{1}{p}$ and $
\frac{1}{p}-\frac{1}{p_2}+\frac{s_2-|\beta|}{n}>0$ hold true). So
 we may assume $|\beta|=s$. Since $|\alpha|\leq s$ and $\beta\leq \alpha$, this is possible only if
  $|\alpha|=s$ and $\beta=\alpha$. Exactly similar to [(i)], one can
  show that $r=\infty$ and $q=p$ satisfy (\ref{eqnholderembed1}).
\item \textbf{$s_1-s=n(\frac{1}{p_1}-\frac{1}{p})$, $s_2-s=
n(\frac{1}{p_2}-\frac{1}{p})$:} If $|\alpha-\beta|<s$,
$|\beta|<s$, then the proof of \textbf{Case 1} works. If
$|\alpha-\beta|=s$ and $|\beta|<s$, then the argument given in
item [(i)] works. If $|\alpha-\beta|<s$ and $|\beta|=s$, then the
argument given in item [(ii)] works. Also note that, since
$|\alpha|\leq s$ and $\beta\leq \alpha$, it is not possible to
have $|\alpha-\beta|=|\beta|=s$.
\end{enumerate}
So we proved $\|
 uv\|_{s,p}\preceq \| u\|_{s_1,p_1}\| v\|_{s_2,p_2}$ for $v\in C_c^{\infty}(\mathbb{R}^n)$ and
$u\in W^{s_1,p_1}(\mathbb{R}^n)$. Now suppose $v$ is an arbitrary
element of $W^{s_2,p_2}(\mathbb{R}^n)$. There exists a sequence
$v_j \in C_c^{\infty}(\mathbb{R}^n)$ such that $v_j \rightarrow v$
in $W^{s_2,p_2}(\mathbb{R}^n)$. We have
\begin{equation*}
\| u v_j - u v_{j'} \|_{s,p}\preceq \|
v_j-v_{j'}\|_{s_2,p_2}\| u\|_{s_1,p_1}
\end{equation*}
Therefore $u v_j$ is a Cauchy sequence in $W^{s,p}(\mathbb{R}^n)$
 and so $u v_j$ converges to an element $w\in W^{s,p}(\mathbb{R}^n)$.
 Since
 $W^{s,p}(\mathbb{R}^n)\hookrightarrow L^{p}(\mathbb{R}^n)$, $u v_j  \rightarrow
 w$ in $L^{p}(\mathbb{R}^n)$. Hence, there exists a subsequence $u \tilde{v}_j
 $ that converges to $w$ almost everywhere. On the other hand,
\begin{align*}
\tilde{v}_j\rightarrow v \quad \textrm{in
$W^{s_2,p_2}(\mathbb{R}^n)$}&\Longrightarrow \tilde{v}_j
\rightarrow v \quad \textrm{in
$L^{p_2}(\mathbb{R}^n)$}\\&\Longrightarrow \exists\, \textrm{a
subsequence $\tilde{\tilde{v}}_j$ such that}\,
\tilde{\tilde{v}}_j\rightarrow v \,\, a.e.
\end{align*}
Consequently $u \tilde{\tilde{v}}_j  \rightarrow uv \, a.e.$ and
$u \tilde{\tilde{v}}_j \rightarrow w\, a.e.$, and so $uv=w\,\,
a.e.$ as well. Therefore, $uv\in W^{s,p}(\mathbb{R}^n)$ and
\begin{align*}
\| uv\|_{s,p}= \| \lim_{j\rightarrow
\infty}(u v_j )\|_{s,p}=\lim_{j\rightarrow \infty}\|
 (u v_j )\|_{s,p}&\preceq \lim_{j\rightarrow \infty}
 \| v_j\|_{s_2,p_2}\|
 u\|_{s_1,p_1}\\
 &=\| v\|_{s_2,p_2}\|
 u\|_{s_1,p_1}.
\end{align*}
\end{proof}

\begin{corollary}
Using extension operators, one can easily show that the above
result holds also for Sobolev spaces on any bounded domain with
Lipschitz continuous boundary. Indeed, if $P_1:
W^{s_1,p_1}(\Omega)\rightarrow W^{s_1,p_1}(\reals^n)$ and $P_2:
W^{s_2,p_2}(\Omega)\rightarrow W^{s_2,p_2}(\reals^n)$ are
extension operators, then $(P_1 u)(P_2 v)|_{\Omega}=uv$ and
therefore
\begin{align*}
\| uv\|_{W^{s,p}(\Omega)}\leq \| (P_1 u)(P_2
v)\|_{W^{s,p}(\reals^n)} &\preceq \| P_1
u\|_{W^{s_1,p_1}(\reals^n)}\| P_2
v\|_{W^{s_2,p_2}(\reals^n)}\\
&\preceq \| u\|_{W^{s_1,p_1}(\Omega)}\|
v\|_{W^{s_2,p_2}(\Omega)}\,.
\end{align*}
\end{corollary}


\section{Sufficient Conditions for $W^{s_1,p_1}\times W^{s_2,p_2}\hookrightarrow
W^{s,p}$,
         $s\geq 0$, $s \in \mathbb{R}$}
\label{sec:positives}

As noted earlier in Remark~\ref{rem:1} just following
Theorem~\ref{thm4.6}, on that theorem $p_i$ was not required to
be less than or equal to $p$. It is the restriction that $s$ be
an integer in Theorem~\ref{thm4.6} that makes it possible to
remove the ordering between $p_i$ and $p$. We see in
Theorem~\ref{thm4.2} below that alternatively, one can restrict
consideration to a bounded domain $\Omega$ in place of
$\mathbb{R}^n$, allowing $s$ to be noninteger, yet still removing
the ordering restriction between $p_i$ and $p$. First we consider
the case of unbounded domains and real exponents, with the
ordering restriction between $p_i$ and $p$. It is worth
mentioning that, as opposed to the proofs of the similar results
in the literature which are based on Littlewood-Paley theory and
Besov spaces, the proofs presented here are based on
interpolation theory and embedding theorems without any reference
to Littlewood-Paley theory.

Before proceeding any further, first we need to state two lemmas:
\begin{lemma}\lab{lemessential1}
Let $\Omega$ be a bounded open subset of $\mathbb{R}^n$ with
Lipschitz continuous boundary, or $\Omega=\mathbb{R}^n$.
\begin{align*}
& \forall \, \epsilon>0, \quad \forall \, m\in
[0,\frac{n}{p}]\cap\mathbb{Z}, \quad W^{m,p}(\Omega)\times
W^{\frac{n}{p}+\epsilon,p}(\Omega)\hookrightarrow
W^{m,p}(\Omega).\\
&\forall \, \epsilon>0, \quad \forall \, m\in
[0,\frac{n}{p}]\cap\mathbb{Z}, \quad
W^{\frac{n}{p}+\epsilon,p}(\Omega)\times
W^{m,p}(\Omega)\hookrightarrow W^{m,p}(\Omega).
\end{align*}
\end{lemma}
\begin{proof}{\bf (Lemma~\ref{lemessential1})}
This is a direct consequence of the previous theorem.
\end{proof}
\begin{lemma}\lab{lemessential2}
Let $\Omega$ be a bounded open subset of $\mathbb{R}^n$ with
Lipschitz continuous boundary, or $\Omega=\mathbb{R}^n$.
\begin{align*}
& \forall \, \epsilon>0, \quad \forall \, s\in [0,\frac{n}{p}],
\quad W^{s,p}(\Omega)\times
W^{\frac{n}{p}+\epsilon,p}(\Omega)\hookrightarrow
W^{s,p}(\Omega).\\
&\forall \, \epsilon>0, \quad \forall \, s\in [0,\frac{n}{p}],
\quad W^{\frac{n}{p}+\epsilon,p}(\Omega)\times
W^{s,p}(\Omega)\hookrightarrow W^{s,p}(\Omega).
\end{align*}
\end{lemma}
\begin{proof}{\bf (Lemma~\ref{lemessential2})}
Clearly we just need to prove the first statement. Let
$\epsilon>0$ and $s\in [0,\frac{n}{p}]$ be given. By Lemma
\ref{lemessential1} if $s\in \mathbb{Z}$ the claim holds true. So
we may assume $s\not \in \mathbb{Z}$. Since
$\frac{n}{p}+\epsilon>\frac{n}{p}$,
$W^{\frac{n}{p}+\epsilon,p}(\Omega)$ is an algebra and
\begin{equation}\lab{eqnsobolevalgebra12}
W^{\frac{n}{p}+\epsilon,p}(\Omega)\times
W^{\frac{n}{p}+\epsilon,p}(\Omega)\hookrightarrow
W^{\frac{n}{p}+\epsilon,p}(\Omega).
\end{equation}
Also $W^{\frac{n}{p}+\epsilon,p}(\Omega)\hookrightarrow
L^{\infty}(\Omega)$. Hence
\begin{equation}\lab{eqnlebesgue2}
W^{\frac{n}{p}+\epsilon,p}(\Omega)\times
W^{0,p}(\Omega)\hookrightarrow W^{0,p}(\Omega).\quad
(L^{\infty}\times L^{p}\hookrightarrow L^{p})
\end{equation}
Let $\theta=\frac{s}{\frac{n}{p}+\epsilon}$; clearly
$0<\theta<1$. Let $p_1=1$ (so if we let
$\frac{1}{r}=\frac{1}{p_1}+\frac{1}{p}-1$, then $r=p$). We want
to use real interpolation between (\ref{eqnsobolevalgebra12}) and
(\ref{eqnlebesgue2}). By Theorem \ref{thm3.03} we have
\begin{equation*}
(W^{\frac{n}{p}+\epsilon,p}(\Omega),
W^{\frac{n}{p}+\epsilon,p}(\Omega))_{\theta,p_1}\times
(W^{0,p}(\Omega),
W^{\frac{n}{p}+\epsilon,p}(\Omega))_{\theta,p}\hookrightarrow
(W^{0,p}(\Omega), W^{\frac{n}{p}+\epsilon,p}(\Omega))_{\theta,r}.
\end{equation*}
By Theorem \ref{thm3.01} we have
\begin{equation*}
(W^{\frac{n}{p}+\epsilon,p}(\Omega),
W^{\frac{n}{p}+\epsilon,p}(\Omega))_{\theta,p_1}=W^{\frac{n}{p}+\epsilon,p},\quad
(W^{0,p}(\Omega),
W^{\frac{n}{p}+\epsilon,p}(\Omega))_{\theta,p}=W^{s,p}.\quad
(s\not \in \mathbb{Z})
\end{equation*}
Hence
\begin{equation*}
W^{\frac{n}{p}+\epsilon,p}(\Omega)\times
W^{s,p}(\Omega)\hookrightarrow W^{s,p}(\Omega).
\end{equation*}
\end{proof}
\begin{theorem}[Multiplication theorem for Sobolev spaces on the whole space, nonnegative exponents]\lab{thm4.1}
Assume $s_i,s$ and $1 \leq p_i \leq p< \infty$ ($i=1,2$) are real
numbers satisfying
\begin{enumerate}[(i)]
\item  $s_i \geq s$ 
\item  $s\geq 0$,
\item  $s_i-s\geq n(\dfrac{1}{p_i}-\dfrac{1}{p})$,
\item  $s_1+s_2-s>n(\dfrac{1}{p_1}+\dfrac{1}{p_2}-\dfrac{1}{p})$.
\end{enumerate}
Claim: If $u\in W^{s_1,p_1}(\mathbb{R}^n)$ and $v\in
W^{s_2,p_2}(\mathbb{R}^n)$, then $uv \in W^{s,p}(\mathbb{R}^n)$
and moreover the pointwise multiplication of functions is a
continuous bilinear map
\begin{equation*}
W^{s_1,p_1}(\mathbb{R}^n)\times
W^{s_2,p_2}(\mathbb{R}^n)\rightarrow W^{s,p}(\mathbb{R}^n).
\end{equation*}
\end{theorem}
\begin{proof}{\bf (Theorem~\ref{thm4.1})}
First we consider the special case where $p_1=p_2=p$ and then we
will prove the general case.
\begin{itemizeXALI}
\item \textbf{Step 1 $p_1=p_2=p$:} In this case the assumptions can be
rewritten as follows:
\begin{equation*}
s_1,s_2\geq s \geq 0, \quad s_1+s_2-s>\frac{n}{p}.
\end{equation*}
\begin{itemizeXXALI}
\item \textbf{Case 1 $s>\frac{n}{p}$:} If $s>\frac{n}{p}$, then
 $W^{s,p}(\mathbb{R}^n)$is an algebra and therefore we can write
\begin{align*}
W^{s_1,p_1}(\mathbb{R}^n) \times
W^{s_2,p_2}(\mathbb{R}^n)&\hookrightarrow
W^{s,p}(\mathbb{R}^n)\times
W^{s,p}(\mathbb{R}^n)\quad (\textrm{by assumption $s_1,s_2\geq s$})\\
&\hookrightarrow W^{s,p}(\mathbb{R}^n).
\end{align*}
\item \textbf{Case 2 $s\leq \frac{n}{p}$:} By Lemma
\ref{lemessential2} for all $\epsilon>0$
\begin{align*}
&W^{s,p}(\mathbb{R}^n)\times
W^{\frac{n}{p}+\epsilon,p}(\mathbb{R}^n)\hookrightarrow
W^{s,p}(\mathbb{R}^n),\\
&W^{\frac{n}{p}+\epsilon,p}(\mathbb{R}^n) \times
W^{s,p}(\mathbb{R}^n)\hookrightarrow W^{s,p}(\mathbb{R}^n).
\end{align*}
In particular for $\epsilon=s_1+s_2-s-\frac{n}{p}>0$ we have
\begin{align*}
&W^{s,p}(\mathbb{R}^n)\times
W^{s_1+s_2-s,p}(\mathbb{R}^n)\hookrightarrow
W^{s,p}(\mathbb{R}^n),\\
&W^{s_1+s_2-s,p}(\mathbb{R}^n) \times
W^{s,p}(\mathbb{R}^n)\hookrightarrow W^{s,p}(\mathbb{R}^n).
\end{align*}
We may consider the following cases:
\begin{enumerate}[(i)]
\item \textbf{$p< 2, s_1,s_2\not \in \mathbb{Z}$:} Let $\frac{1}{r}=\frac{1}{p}+\frac{1}{p}-1>0$.
Let $\theta$ be such that $(1-\theta)s+\theta(s_1+s_2-s)=s_1$. As
it was discussed in the proof of Theorem \ref{thm4.0}, for this
$\theta$, $(1-\theta)(s_1+s_2-s)+\theta s=s_2$. By Theorem
\ref{thm3.03} we have
\begin{align*}
(W^{s,p}(\mathbb{R}^n),W^{s_1+s_2-s,p}(\mathbb{R}^n))_{\theta,p}\times
(W^{s_1+s_2-s,p}(\mathbb{R}^n) &
,W^{s,p}(\mathbb{R}^n))_{\theta,p}\\
&\hookrightarrow
(W^{s,p}(\mathbb{R}^n),W^{s,p}(\mathbb{R}^n))_{\theta,r}.
\end{align*}
Consequently, since $s_1\not \in \mathbb{Z}$ and $s_2\not \in
\mathbb{Z}$,
\begin{equation*}
W^{s_1,p}(\mathbb{R}^n)\times
W^{s_2,p}(\mathbb{R}^n)\hookrightarrow W^{s,p}(\mathbb{R}^n).
\end{equation*}
\item \textbf{$p< 2, s_1\in \mathbb{Z},\, s_2\not \in \mathbb{Z}$:}
If $s_1=s$, then from $s_1+s_2-s>\frac{n}{p}$ it follows that
$s_2>\frac{n}{p}$. So in this case the claim reduces to what was
proved in Lemma \ref{lemessential2}. If $s_1\neq s$, let
$\tilde{s}_1=s_1-\epsilon$ where
\begin{equation*}
\epsilon=
\frac{1}{2}\min(s_1-\floor{s_1},s_1-s,s_1+s_2-s-\frac{n}{p})>0.
\end{equation*}
Clearly,
\begin{equation*}
\tilde{s}_1\not \in \mathbb{Z}, \quad \tilde{s}_1\geq s,\quad
s_2\geq s, \quad \tilde{s}_1+s_2-s>\frac{n}{p}.
\end{equation*}
Therefore, by what was proved in the previous case
\begin{equation*}
W^{\tilde{s}_1,p}\times W^{s_2,p}\hookrightarrow W^{s,p}.
\end{equation*}
Now the claim follows from the fact that
$W^{s_1,p}\hookrightarrow W^{\tilde{s}_1,p}$.
\item \textbf{$p< 2, s_1\not \in \mathbb{Z},\, s_2 \in \mathbb{Z}$:}
Just switch the roles of $s_1$ and $s_2$ in the previous case.
\item \textbf{$p< 2, s_1 \in \mathbb{Z},\, s_2 \in \mathbb{Z}$:} Note
that both of $s_1$ and $s_2$ cannot be equal to $s$ because
$s_1+s_2-s>\frac{n}{p}$ but $s\leq \frac{n}{p}$. Because of the
symmetry in the roles of $s_1$ and $s_2$, without loss of
generality we may assume that $s_1\neq s$. let
$\tilde{s}_1=s_1-\epsilon$ where
\begin{equation*}
\epsilon=
\frac{1}{2}\min(s_1-\floor{s_1},s_1-s,s_1+s_2-s-\frac{n}{p})>0\,.
\end{equation*}
Clearly
\begin{equation*}
\tilde{s}_1\not \in \mathbb{Z}, \quad \tilde{s}_1\geq s,\quad
s_2\geq s, \quad \tilde{s}_1+s_2-s>\frac{n}{p}\,.
\end{equation*}
and so the problem reduces to the previous case.

At this point we are done with the case $p<2$.
\item \textbf{$p\geq 2, s\not \in \mathbb{Z},\, s_1+s_2-s \not \in
\mathbb{Z}$:} This time we use complex interpolation. Define
$\theta$ as before. By Theorem \ref{thm3.03} we have
\begin{align*}
[W^{s,p}(\mathbb{R}^n),W^{s_1+s_2-s,p}(\mathbb{R}^n)]_{\theta}\times
[W^{s_1+s_2-s,p}(\mathbb{R}^n)&,W^{s,p}(\mathbb{R}^n)]_{\theta}\\&\hookrightarrow
[W^{s,p}(\mathbb{R}^n),W^{s,p}(\mathbb{R}^n)]_{\theta}\,.
\end{align*}
Since $s$ and $s_1+s_2-s$ are not integers and $p\geq 2$ (see
Theorem \ref{thm3.02}),
\begin{align*}
&W^{s_1,p}(\mathbb{R}^n)\hookrightarrow
[W^{s,p}(\mathbb{R}^n),W^{s_1+s_2-s,p}(\mathbb{R}^n)]_{\theta}\,,\\
&W^{s_2,p}(\mathbb{R}^n)\hookrightarrow
[W^{s_1+s_2-s,p}(\mathbb{R}^n),W^{s,p}(\mathbb{R}^n)]_{\theta}\,.
\end{align*}
Consequently
\begin{equation*}
W^{s_1,p}(\mathbb{R}^n)\times
W^{s_2,p}(\mathbb{R}^n)\hookrightarrow W^{s,p}(\mathbb{R}^n)\,.
\end{equation*}
\item \textbf{$p\geq 2, s\not \in \mathbb{Z},\, s_1+s_2-s \in
\mathbb{Z}$:} Both of $s_1$ and $s_2$ cannot be equal to $s$
because $s_1+s_2-s>\frac{n}{p}$ but $s\leq \frac{n}{p}$. Because
of the symmetry in the roles of $s_1$ and $s_2$, without loss of
generality we may assume that $s_1\neq s$. Let
$\tilde{s}_1=s_1-\epsilon$ where
\begin{equation*}
\epsilon=\frac{1}{2}\min(1,s_1-s,s_1+s_2-s-\frac{n}{p}).
\end{equation*}
Clearly
\begin{align*}
&\tilde{s}_1\geq s,\quad
\tilde{s}_1+s_2-s=s_1+s_2-s-\epsilon>\frac{n}{p},\\
& \tilde{s}_1+s_2-s=s_1+s_2-s-\epsilon \not \in \mathbb{Z}\quad
(\textrm{because $\epsilon\leq \frac{1}{2}$}).
\end{align*}
So by what was proved in the previous case we have
\begin{equation*}
W^{\tilde{s}_1,p}(\mathbb{R}^n)\times
W^{s_2,p}(\mathbb{R}^n)\hookrightarrow W^{s,p}(\mathbb{R}^n).
\end{equation*}
and since $W^{s_1,p}\hookrightarrow W^{\tilde{s}_1,p}$
\begin{equation*}
W^{s_1,p}(\mathbb{R}^n)\times
W^{s_2,p}(\mathbb{R}^n)\hookrightarrow W^{s,p}(\mathbb{R}^n).
\end{equation*}
\item \textbf{$p\geq 2, s \in \mathbb{Z},\, s_1+s_2-s \not \in
\mathbb{Z}$:} If $s_1=s$ or $s_2=s$, the claim follows from Lemma
\ref{lemessential2}. So we may assume $s_1, s_2> s$. Let
$\tilde{s}=s+\epsilon$ where
\begin{align*}
&\epsilon=\frac{1}{2}\min(s_1-s, s_2-s,
s_1+s_2-s-\floor{s_1+s_2-s}, s_1+s_2-s-\frac{n}{p})\\
&(\textrm{note that $s_1+s_2-s-\floor{s_1+s_2-s}<1$})
\end{align*}
Clearly $\tilde{s}$ and $s_1+s_2-\tilde{s}$ are not integers and
\begin{equation*}
s_1\geq \tilde{s},\quad s_2\geq \tilde{s},\quad
s_1+s_2-\tilde{s}>\frac{n}{p}.
\end{equation*}
So by what was proved in previous cases
\begin{equation*}
W^{s_1,p}(\mathbb{R}^n)\times
W^{s_2,p}(\mathbb{R}^n)\hookrightarrow
W^{\tilde{s},p}(\mathbb{R}^n)=W^{s+\epsilon,p}(\mathbb{R}^n)\hookrightarrow
W^{s,p}(\mathbb{R}^n) .
\end{equation*}
\item \textbf{$p\geq 2, s \in \mathbb{Z},\, s_1+s_2-s \in
\mathbb{Z}$:} If $s_1=s$ or $s_2=s$, the claim follows from Lemma
\ref{lemessential2}. So we may assume $s_1, s_2> s$. Let
$\tilde{s}=s+\epsilon$ where
\begin{equation*}
\epsilon=\frac{1}{2}(1,s_1-s,s_2-s,s_1+s_2-s-\frac{n}{p}).
\end{equation*}
We have $\epsilon\leq \frac{1}{2}$, so $\tilde{s}$ and
$s_1+s_2-\tilde{s}$ are not integers. Also clearly
\begin{equation*}
s_1\geq \tilde{s},\quad s_2\geq \tilde{s},\quad
s_1+s_2-\tilde{s}>\frac{n}{p}.
\end{equation*}
So by what was proved in previous cases
\begin{equation*}
W^{s_1,p}(\mathbb{R}^n)\times
W^{s_2,p}(\mathbb{R}^n)\hookrightarrow
W^{\tilde{s},p}(\mathbb{R}^n)=W^{s+\epsilon,p}(\mathbb{R}^n)\hookrightarrow
W^{s,p}(\mathbb{R}^n) .
\end{equation*}
\end{enumerate}
\end{itemizeXXALI}
\item \textbf{Step 2: General Case} This step is exactly the same as
step 2 in the proof of Theorem \ref{thm4.0}. We just need to
replace every occurrence of $H^{r,q}(\mathbb{R}^n)$ with
$W^{r,q}(\mathbb{R}^n)$.
\end{itemizeXALI}
\end{proof}
Proposition \ref{propholderwrong} shows that the claim of Theorem
\ref{thm4.1} does not necessarily hold if one removes the
assumption $p_i\leq p$. Of course, the next theorem shows that the
assumption $p_i\leq p$ is not necessary on bounded domains.


\begin{theorem}[Multiplication theorem for Sobolev spaces on bounded domains, nonnegative exponents]\lab{thm4.2}
Let $\Omega$ be a bounded domain in $\mathbb{R}^n$ with Lipschitz
continuous boundary. Assume $s_i,s$ and $1 \leq p_i, p< \infty$
($i=1,2$) are real numbers satisfying
\begin{enumerate}[(i)]
\item  $s_i \geq s$ 
\item  $s\geq 0$,
\item  $s_i-s\geq n(\dfrac{1}{p_i}-\dfrac{1}{p})$,
\item  $s_1+s_2-s>n(\dfrac{1}{p_1}+\dfrac{1}{p_2}-\dfrac{1}{p})$.
\end{enumerate}
In the case where $\max\{p_1, p_2\}>p$ instead of (iv)
assume that $s_1+s_2-s>\frac{n}{\min\{p_1,p_2\}}$ and that the inequalities in (i) and (iii) are strict.\\
Claim: If $u\in W^{s_1,p_1}(\Omega)$ and $v\in
W^{s_2,p_2}(\Omega)$, then $uv \in W^{s,p}(\Omega)$ and moreover
the pointwise multiplication of functions is a continuous
bilinear map
\begin{equation*}
W^{s_1,p_1}(\Omega)\times W^{s_2,p_2}(\Omega)\rightarrow
W^{s,p}(\Omega).
\end{equation*}
\end{theorem}
\begin{proof}{\bf (Theorem~\ref{thm4.2})}
\leavevmode
\begin{itemizeXALI}
\item \textbf{Step 1 $p_1=p_2=p$:} By the (exact) same proof as the one given in Step 1 of the proof of
Theorem \ref{thm4.1} we have
\begin{equation*}
W^{s_1,p}(\Omega)\times W^{s_2,p}(\Omega)\hookrightarrow
W^{s,p}(\Omega),
\end{equation*}
provided $s_1,s_2\geq s$ and $s_1+s_2-s>\frac{n}{p}$.
\item \textbf{Step 2:} Now we prove the
general case. Because of the symmetry in the roles of $p_1$ and
$p_2$ without loss of generality we may assume $p_2\leq p_1$. We
may consider three cases:
\begin{itemizeXXALI}
\item \textbf{Case 1 $p_2 \leq p_1 \leq p$:} The proof is exactly the
same as the one presented in Step 2 of Theorem \ref{thm4.1}.
\item \textbf{Case 2 $p_2 \leq p < p_1$:} Let $\tilde{p}_1=p$. Since, by assumption, $s_1+s_2-s>\frac{n}{p_2}$, we can
choose $\tilde{s}_1\in (s,s_1)$ such that $\tilde{s}_1+s_2-s>\frac{n}{p_2}$. Now one can easily check that the tuple $(\tilde{s}_1,s_2,s,\tilde{p}_1,p_2,p)$
also satisfies all the assumptions of the theorem. So
by what was proved in the previous case we have
\begin{equation*}
W^{\tilde{s}_1,\tilde{p_1}}(\Omega)\times
W^{s_2,p_2}(\Omega)\hookrightarrow W^{s,p}(\Omega).
\end{equation*}
By the third embedding theorem (Theorem \ref{thm3.4})
$W^{s_1,p_1}(\Omega)\hookrightarrow W^{\tilde{s}_1,\tilde{p_1}}(\Omega)$
(because $s_1-\tilde{s}_1>0> \frac{n}{p_1}-\frac{n}{\tilde{p}_1}$). Hence
\begin{equation*}
W^{s_1,p_1}(\Omega)\times W^{s_2,p_2}(\Omega)\hookrightarrow
W^{s,p}(\Omega).
\end{equation*}
\item \textbf{Case 3 $p<p_2\leq p_1$:} Let
\begin{equation*}
\epsilon=\frac{1}{6}\min\{s_1-s,s_2-s,s_1+s_2-s-\frac{n}{p_2}\}
\end{equation*}
and set $\hat{s}_1=s_1-\epsilon$, $\hat{s}_2=s_2-\epsilon$, and $\hat{s}=s+\epsilon$. Clearly $\hat{s}_1,\hat{s}_2\geq \hat{s}$. Moreover,
Since $s_1+s_2-s>\frac{n}{p_2}$, we have
\begin{equation*}
\hat{s}_1+\hat{s}_2-\hat{s}=s_1+s_2-s-3\epsilon> \frac{n}{p_2}
\end{equation*}
Thus, by what was proved in Step 1 we have
\begin{equation*}
W^{\hat{s}_1,p_2}(\Omega)\times W^{\hat{s}_2,p_2}(\Omega)\hookrightarrow
W^{\hat{s},p_2}(\Omega).
\end{equation*}
Now note that $p_1\geq p_2$ and $p_2>p$, so by the third
embedding theorem (Theorem \ref{thm3.4})
$W^{s_1,p_1}(\Omega)\hookrightarrow W^{\hat{s}_1,p_2}(\Omega)$ and
$W^{\hat{s},p_2}(\Omega)\hookrightarrow W^{s,p}(\Omega)$. Therefore
\begin{equation*}
W^{s_1,p_1}(\Omega)\times W^{s_2,p_2}(\Omega)\hookrightarrow
W^{s,p}(\Omega).
\end{equation*}
\end{itemizeXXALI}
\end{itemizeXALI}
\end{proof}


\section{Sufficient Conditions for $W^{s_1,p_1}\times W^{s_2,p_2}\hookrightarrow
W^{s,p}$,
         $s<0$, $s \in \mathbb{R}$}
\label{sec:negatives}
\begin{theorem}[Multiplication theorem for Sobolev spaces on the whole space, negative exponents I]\lab{thm4.3}
Assume $s_i, s$ and $1 < p_i \leq p < \infty$ ($i=1,2$) are real
numbers satisfying
\begin{enumerate}[(i)]
\item  $s_i \geq s$, 
\item  $\min\{s_1, s_2\}<0$,
\item  $s_i-s\geq n(\dfrac{1}{p_i}-\dfrac{1}{p})$,
\item  $s_1+s_2-s>n(\dfrac{1}{p_1}+\dfrac{1}{p_2}-\dfrac{1}{p})$.
\item $s_1+s_2\geq n(\dfrac{1}{p_1}+\dfrac{1}{p_2}-1)\geq 0$.
\end{enumerate}
Then the pointwise multiplication of functions extends uniquely
to a continuous bilinear map
\begin{equation*}
W^{s_1,p_1}(\mathbb{R}^n)\times
W^{s_2,p_2}(\mathbb{R}^n)\rightarrow W^{s,p}(\mathbb{R}^n).
\end{equation*}
\end{theorem}
\begin{proof}{\bf (Theorem~\ref{thm4.3})}
 Since by assumption $s_1+s_2\geq 0$, $s_1$ and $s_2$ cannot
both be negative. WLOG we can assume $s_1$ is negative and $s_2$
is positive. Also note that by assumption $s \leq s_1$ so $s$ is
also negative.

Note that $C_c^{\infty}$ is dense in all Sobolev spaces on
$\reals^n$.
Considering this, first we prove that for $u\in W^{s_1,p_1},
\varphi\in C_c^{\infty}$
\begin{equation}\lab{eqdualproduct1}
\|u\varphi\|_{s,p}\preceq \|u\|_{s_1,p_1}\|\varphi\|_{s_2,p_2} .
\end{equation}
Note that
\begin{equation*}
\|f\|_{s,p}=\sup_{\psi \in C_c^{\infty}}\frac{|\langle
f,\psi\rangle_{W^{s,p}\times W^{-s,p'}}|}{\|\psi\|_{-s,p'}}.
\end{equation*}
Thus we just need to show that
\begin{equation*}
|\langle u\varphi, \psi\rangle_{W^{s,p}\times W^{-s,p'}}|\preceq
\|u\|_{s_1,p_1}\|\varphi\|_{s_2,p_2}\|\psi\|_{-s,p'}.
\end{equation*}
We have

\begin{equation*}
|\langle u\varphi, \psi \rangle_{W^{s,p}\times
W^{-s,p'}}|=|\langle u,\varphi\psi\rangle_{W^{s_1,p_1}\times
W^{-s_1,p_1'}}|\preceq \|u\|_{s_1,p_1}\|\varphi\psi\|_{-s_1,p_1'}.
\end{equation*}
Note that the first equality holds true because duality pairing
is an extension of the $L^2$ inner product of smooth functions. Now
it is enough to prove that
\begin{equation*}
\|\varphi\psi\|_{-s_1,p_1'}\preceq
\|\varphi\|_{s_2,p_2}\|\psi\|_{-s,p'}.
\end{equation*}
$-s_1, s_2, -s$ are all nonnegative. So, by Theorem \ref{thm4.1},
in order to ensure that the above inequality is true we just need
to check the followings:
\begin{align*}
&p'\leq p_1' \quad (\textrm{true because $p_1\leq p$}),\quad p_2\leq p_1'\quad (\textrm{true because $\frac{1}{p_1}+\frac{1}{p_2}\geq 1$}),\\
 &-s_1\leq -s \quad (\textrm{true because $s\leq s_1$}), \quad
 s_2\geq -s_1 \quad (\textrm{true because $s_1+s_2\geq 0$}),\\
& s_2-(-s_1) \geq n(\frac{1}{p_2}-\frac{1}{p'_1})\quad
(\textrm{true because $s_2+s_1 \geq
n(\frac{1}{p_2}+\frac{1}{p_1}-1)$}), \\
& -s-(-s_1)\geq n(\frac{1}{p'}-\frac{1}{p'_1}) \quad
(\textrm{true because $s_1-s \geq
n(\frac{1}{p_1}-\frac{1}{p})$}),\\
&s_2+(-s)-(-s_1)>
n(\frac{1}{p_2}+\frac{1}{p'}-\frac{1}{p'_1})\quad (\textrm{true
because $s_1+s_2-s> n(\frac{1}{p_2}+\frac{1}{p_1}-\frac{1}{p})$}).
\end{align*}
 Therefore the inequality
(\ref{eqdualproduct1}) holds for $u\in W^{s_1,p_1}$ and $\varphi
\in C_c^{\infty}$. To prove the general case we proceed as
follows: Suppose $v\in W^{s_2,p_2}$. There exists a sequence
$\varphi_k\in C_c^{\infty}$ such that $\varphi_k\rightarrow v$ in
$W^{s_2,p_2}$. Since $s_2\geq 0$, $W^{s_2,p_2}\hookrightarrow
L^{p_2}$ and therefore $\varphi_k\rightarrow v$ in $L^{p_2}$.
Consequently, by possibly passing to a subsequence,
$\varphi_k\rightarrow v$ a.e. which implies that
$u\varphi_k\rightarrow uv$ a.e..

On the other hand we have
\begin{equation*}
\|u(\varphi_i-\varphi_j)\|_{s,p}\preceq
\|u\|_{s_1,p_1}\|\varphi_i-\varphi_j\|_{s_2,p_2}.
\end{equation*}
It follows that $u\varphi_k$ is a Cauchy sequence in $W^{s,p}$ and
thus it is convergent to some function $w\in W^{s,p}$. Since
$u\varphi_k\rightarrow uv$ a.e., we can conclude that $w=uv$, that
is, $u\varphi_k\rightarrow uv$ in $W^{s,p}$. Finally
\begin{equation*}
\forall\, k\quad \|u\varphi_k\|_{s,p}\preceq
\|u\|_{s_1,p_1}\|\varphi_k\|_{s_2,p_2}, \quad
\end{equation*}
and so by passing to the limit as $k\rightarrow \infty$
\begin{equation*}
\|uv\|_{s,p}\preceq \|u\|_{s_1,p_1}\|v\|_{s_2,p_2}.
\end{equation*}
\end{proof}
\begin{theorem}[Multiplication theorem for Sobolev spaces on the whole space, negative exponents II]\lab{thm4.5}
Assume $s_i,s$ and $1 < p,  p_i < \infty$ ($i=1,2$) are real
numbers satisfying
\begin{enumerate}[(i)]
\item  $s_i \geq s$, 
\item  $\min\{s_1, s_2\}\geq0$ and $s<0$,
\item  $s_i-s\geq n(\dfrac{1}{p_i}-\dfrac{1}{p})$,
\item  $s_1+s_2-s>n(\dfrac{1}{p_1}+\dfrac{1}{p_2}-\dfrac{1}{p})\geq
0$.
\item $s_1+s_2> n(\dfrac{1}{p_1}+\dfrac{1}{p_2}-1)$. \quad
(\textrm{the inequality is strict})
\end{enumerate}
Then the pointwise multiplication of functions extends uniquely
to a continuous bilinear map
\begin{equation*}
W^{s_1,p_1}(\reals^n)\times W^{s_2,p_2}(\reals^n)\rightarrow
W^{s,p}(\reals^n).
\end{equation*}
\end{theorem}
\begin{proof}{\bf (Theorem~\ref{thm4.5})}
Let $\epsilon>0$ be such that
\begin{equation*}
\epsilon<\frac{1}{n}\min\{s_1+s_2-s-(\frac{n}{p_1}+\frac{n}{p_2}-\frac{n}{p}),
s_1+s_2-(\frac{n}{p_1}+\frac{n}{p_2}-n) \}
\end{equation*}
Let
\begin{equation*}
\frac{1}{r}=\max\{\frac{1}{p_1}-\frac{s_1}{n},
\frac{1}{p_2}-\frac{s_2}{n},
\frac{1}{p_1}-\frac{s_1}{n}+\frac{1}{p_2}-\frac{s_2}{n}+\epsilon,\frac{1}{p}\}.
\end{equation*}
Note that $r>0$ because $\frac{1}{r}\geq\frac{1}{p}>0$. Also
$\frac{1}{r}<1$ because each element in the set over which we are
taking the maximum is strictly less than $1$:
\begin{align*}
&\frac{1}{p_1}-\frac{s_1}{n}\leq \frac{1}{p_1}<1, \quad
\frac{1}{p_2}-\frac{s_2}{n}\leq \frac{1}{p_2}<1,\quad \frac{1}{p}<1\\
&\epsilon<
\frac{1}{n}[s_1+s_2-(\frac{n}{p_1}+\frac{n}{p_2}-n)]\Longrightarrow
\frac{1}{p_1}-\frac{s_1}{n}+\frac{1}{p_2}-\frac{s_2}{n}+\epsilon
<1.
\end{align*}
\begin{itemize}
\item \textbf{Claim 1:} $W^{s_1,p_1}(\reals^n)\times
W^{s_2,p_2}(\reals^n)\hookrightarrow L^r(\reals^n)$.
\item \textbf{Claim 2:} $L^r(\reals^n)\hookrightarrow
W^{s,p}(\reals^n)$.
\end{itemize}
Clearly if we prove \textbf{Claim 1} and \textbf{Claim 2}, then we are done.\\
\textbf{Proof of Claim 1}: All the exponents are nonnegative, so
it is enough to check the assumptions of Theorem \ref{thm4.6}.
\begin{align*}
& s_1-0\geq n(\frac{1}{p_1}-\frac{1}{r})\quad (\textrm{true
because $\frac{1}{r}\geq  \frac{1}{p_1}-\frac{s_1}{n}$})\\
& s_2-0\geq n(\frac{1}{p_2}-\frac{1}{r})\quad (\textrm{true
because $\frac{1}{r}\geq  \frac{1}{p_2}-\frac{s_2}{n}$})\\
& s_1+s_2> n(\frac{1}{p_1}+\frac{1}{p_2}-\frac{1}{r})\quad
(\textrm{true because $\frac{1}{r}>
\frac{1}{p_1}-\frac{s_1}{n}+\frac{1}{p_2}-\frac{s_2}{n}$}).
\end{align*}
\textbf{Proof of Claim 2}: We have
$(L^r(\reals^n))^{*}=L^{r'}(\reals^n)$ and
 $(W^{s,p}(\reals^n))^{*}=W^{-s,p'}(\reals^n)$. In what follows we
 will show that $W^{-s,p'}(\reals^n) \hookrightarrow
 L^{r'}(\reals^n)$; then since $W^{-s,p'}(\reals^n)$ is dense in
 $L^{r'}(\reals^n)$ ($C_c^{\infty}(\reals^n)\subseteq W^{-s,p'}(\reals^n)$ and $C_c^{\infty}(\reals^n)$ is dense in
 $L^{r'}(\reals^n)$), we are allowed to take the dual of both sides and it immediately follows that the claim is
 true.

 Note that by the definition of $r$, we have $\frac{1}{p}\leq\frac{1}{r}$ and therefore $p'\leq r'$. So, according to Theorem
 \ref{thm3.2},
 in order to show that $W^{-s,p'}(\reals^n) \hookrightarrow
 L^{r'}(\reals^n)$, it is enough to prove that $-s-\frac{n}{p'}\geq
 0-\frac{n}{r'}$, that is we need to prove that $\frac{1}{p}-\frac{s}{n}\geq
 \frac{1}{r}$. This is true because each element in the set over
 which we are taking the maximum in the definition of
 $\frac{1}{r}$ is less than or equal to $\frac{1}{p}-\frac{s}{n}$:
 \begin{align*}
 & \frac{1}{p}-\frac{s}{n}\geq \frac{1}{p_1}-\frac{s_1}{n}\quad (\textrm{true because $s_1-s\geq n(\frac{1}{p_1}-\frac{1}{p})$})\\
 & \frac{1}{p}-\frac{s}{n}\geq \frac{1}{p_2}-\frac{s_2}{n} \quad (\textrm{true because $s_2-s\geq n(\frac{1}{p_2}-\frac{1}{p})$})\\
 & \frac{1}{p}-\frac{s}{n}\geq \frac{1}{p_1}-\frac{s_1}{n}+\frac{1}{p_2}-\frac{s_2}{n}+\epsilon \quad (\textrm{true because $s_1+s_2-s\geq
 n(\frac{1}{p_1}+\frac{1}{p_2}-\frac{1}{p})+n\epsilon$})\\
 & \frac{1}{p}-\frac{s}{n}\geq \frac{1}{p}\quad (\textrm{true because $s<0$})
 \end{align*}
 \end{proof}
 \begin{remark}
 We note that our earlier article~\cite{HNT07b} contains some
 multiplication results for negative exponents that are similar to
 what we give above as Theorem~\ref{thm4.5}.
 However, a particularly important case is assumption (ii)
 in Theorem~\ref{thm4.5} above, which is a case
 we did not consider in~\cite{HNT07b}.
 \end{remark}

\section*{Acknowledgments}
   \label{sec:ack}

MH was supported in part by NSF Awards~1262982, 1318480, and 1620366.
AB was supported by NSF Award~1262982.

\appendix
\section{A Useful Version of the Multiplication Theorem}
   \label{app:spacesuseful}

In this appendix, using the pointwise multiplication properties
of Besov space, we will prove a very useful multiplication
theorem in Sobolev-Slobodeckij spaces;
this multiplication result is a key tool used in
our recent related article~\cite{holstbehzadan2015} to study
the Einstein constraint equations on asymptotically Euclidean manifolds.
A similar result, but for the case of compact manifolds, was used
in our 2009 article~\cite{HNT07b}.
We include this result to help complete the supporting literature
for the work in both~\cite{HNT07b,holstbehzadan2015}.
\begin{theorem}\lab{lemA13}
Let $s_i \geq s$ with $s_1+s_2\geq 0$, and $1 < p, p_i < \infty$
($i=1,2$) be real numbers satisfying
\begin{align*}
& s_i-s\geq n(\dfrac{1}{p_i}-\dfrac{1}{p}),\quad \textrm{(if $s_i=s \not \in \mathbb{Z}$, then let $p_i\leq p$)}\\
& s_1+s_2-s>n(\dfrac{1}{p_1}+\dfrac{1}{p_2}-\dfrac{1}{p})\geq 0.
\end{align*}
In case $s<0$, in addition let
\begin{equation*}
s_1+s_2> n(\dfrac{1}{p_1}+\dfrac{1}{p_2}-1) \quad
\textrm{(equality is allowed if $min(s_1,s_2)<0$)}.
\end{equation*}
Also in case where $s_1+s_2=0$ and $min(s_1,s_2) \not \in
\mathbb{Z}$, in addition let $\frac{1}{p_1}+\frac{1}{p_2}\geq 1$.
Then the pointwise multiplication of functions extends uniquely
to a continuous bilinear map
\begin{equation*}
W^{s_1,p_1}(\mathbb{R}^n)\times
W^{s_2,p_2}(\mathbb{R}^n)\rightarrow W^{s,p}(\mathbb{R}^n).
\end{equation*}
\end{theorem}
\begin{proof}{\bf (Theorem~\ref{lemA13})}
In this proof we use the notations introduced in the beginning of
Section \ref{subsec:maintheorems}. We may consider three cases:
\begin{itemize}
\item \textbf{Case $1$}: $s\geq 0$.
\item \textbf{Case $2$}: $s<0$ and $min(s_1,s_2)<0$.
\item \textbf{Case $3$}: $s<0$ and $min(s_1,s_2)\geq 0$.
\end{itemize}
In what follows we study each of the above cases separately.
\begin{itemizeXALI}
\item \textbf{Case $1$}: See Theorem \ref{thm4.6} for the case where $s\in
\mathbb{N}_0$; see Theorem \ref{thm4.1} for the case where $p_1,
p_2\leq p$. It remains to prove the claim in the following cases:
\begin{enumerate}[i.]
\item $s_1>s$, $s_2=s$, $s\not \in \mathbb{N}_0$\\
      $p_1>p$, $p_2\leq p$
\item $s_1=s$, $s_2>s$, $s\not \in \mathbb{N}_0$\\
      $p_1\leq p$, $p_2> p$
\item $s_1>s$, $s_2>s$, $s\not \in \mathbb{N}_0$, at least one of $p_1$ or $p_2$ is greater than $p$\\
\end{enumerate}
Proofs of [i] and [ii] are completely similar. Here we only prove
item [i] and item [iii]. \\\\
\textbf{Proof of [i]:} Let
\begin{align*}
&\epsilon:=\frac{1}{4} \min\{s_1-s,
s_1-s-n(\frac{1}{p_1}-\frac{1}{p}),
s_1-n(\frac{1}{p_1}+\frac{1}{p_2}-\frac{1}{p})\}\,.
\end{align*}
We have
\begin{align*}
W^{s_1,p_1}\times W^{s_2,p_2}&\hookrightarrow
B^{s_1-\frac{\epsilon}{2}}_{p_1,p_1}\times
W^{s_2,p_2}\hookrightarrow
B^{s_1-\epsilon}_{p_1,p}\times W^{s_2,p_2}\\
&=B^{s_1-\epsilon}_{p_1,p}\times B^{s_2}_{p_2,p_2}\hookrightarrow
B^s_{p,p}=W^{s,p}\,.
\end{align*}
\textbf{Proof of [iii]:} Let
\begin{align*}
\epsilon:=\frac{1}{4} \min\{s_1-s, s_2-s,
s_1-s-n(\frac{1}{p_1}-\frac{1}{p}),
&s_2-s-n(\frac{1}{p_2}-\frac{1}{p}),\\
&s_1+s_2-s-n(\frac{1}{p_1}+\frac{1}{p_2}-\frac{1}{p})\}\,.
\end{align*}
Let $\tilde{q}_1$, $\tilde{q}_2$, and $\tilde{q}$ be numbers in
$(1,\infty)$ such that $\tilde{q}_1, \tilde{q}_2 \leq \tilde{q}$.
We have
\begin{align*}
W^{s_1,p_1}\times W^{s_2,p_2}\hookrightarrow
B^{s_1-\frac{\epsilon}{2}}_{p_1,p_1}\times
B^{s_2-\frac{\epsilon}{2}}_{p_2,p_2}\hookrightarrow
B^{s_1-\epsilon}_{p_1,\tilde{q}_1}\times
B^{s_2-\epsilon}_{p_2,\tilde{q}_2}\hookrightarrow
B^{s+\epsilon}_{p,\tilde{q}}\hookrightarrow B^s_{p,p}=W^{s,p}\,.
\end{align*}
In the above we have used the well-known embedding theorems for
Besov spaces together with the following multiplication theorem
that is proved in
\cite{jZ77}:\\
\emph{Let $0\leq s\leq s_i$, $1< p_i,p<\infty$, $1< q_i\leq q
<\infty $ $(i=1,2)$ be such that
\begin{align*}
&s_i-s\geq n(\frac{1}{p_i}-\frac{1}{p}),\quad i=1,2\,,\\
 & s_1+s_2-s> n(\frac{1}{p_1}+\frac{1}{p_2}-\frac{1}{p})\,.
\end{align*}
Then for $s\not \in \mathbb{N}$ one has $B^{s_1}_{p_1,q_1}\times
B^{s_2}_{p_2,q_2}\hookrightarrow B^{s}_{p,q}$.}
\item \textbf{Case $2$}: It follows from the argument given in the proof of Theorem \ref{thm4.3} that
without loss of generality we may assume $s_1<0$ and $s_2>0$.
According to the same argument it is enough to prove that
\begin{equation}\lab{importantembeddingproofa}
W^{s_2,p_2}\times W^{-s,p'}\hookrightarrow W^{-s_1,p'_1}\,.
\end{equation}
Since $s_2$, $-s$, and $-s_1$ are all nonnegative numbers, we can
use what was shown in \textbf{Case $1$} to prove the above
embedding. We have
\begin{itemize}
\item $s_1+s_2\geq 0 \Longrightarrow s_2\geq -s_1$.
\item $s_1\geq s \Longrightarrow -s\geq -s_1$.
\item If $s_2=-s_1$ (that is, if $s_1+s_2=0$) and $-s_1\not \in \mathbb{N}$, we must have $p_2\leq
p'_1$, i.e., $1\leq \frac{1}{p_1}+\frac{1}{p_2}$. (holds true by
assumption)
\item If $-s=-s_1$ and $-s_1\not \in \mathbb{N}$, we must have $p'\leq
p_1'$, i.e., $p_1\leq p$. (holds true by assumption)
\item $s_2+s_1\geq
n(\frac{1}{p_1}+\frac{1}{p_2}-1)=n(\frac{1}{p_2}-\frac{1}{p'_1})$.
\item $-s+s_1\geq
n(\frac{1}{p_1}-\frac{1}{p})=n(\frac{1}{p'}-\frac{1}{p'_1})$.
\item
$s_2-s+s_1>n(\frac{1}{p_1}+\frac{1}{p_2}-\frac{1}{p})=n(\frac{1}{p_2}+\frac{1}{p'}-\frac{1}{p'_1})$.
\end{itemize}
So according to what was proved in \textbf{Case $1$}, the
embedding (\ref{importantembeddingproofa}) holds true.
\item For \textbf{Case} $3$, see Theorem \ref{thm4.5}.
\end{itemizeXALI}
\end{proof}
\begin{remark}
Note that in case $s_i=s\not \in \mathbb{Z}$, the condition
$p_i\leq p$ together with $s_i-s\geq
n(\frac{1}{p_i}-\frac{1}{p})$ in fact implies that we must have
$p_i=p$.
\end{remark}

\bibliographystyle{abbrv}
\bibliography{refs}
\end{document}